\documentclass[11pt]{amsart}
\usepackage[margin=1.35in]{geometry}
\usepackage{amscd,amsmath,amsxtra,amsthm,amssymb,stmaryrd,xr,mathrsfs,mathtools,enumerate,commath, comment}
\usepackage{stmaryrd}
\usepackage{xcolor}
\usepackage{commath}
\usepackage{comment}
\usepackage{tikz-cd}
\usepackage{longtable} 
\usepackage{pdflscape} 
\usepackage{booktabs}
\usepackage{hyperref}
\definecolor{vegasgold}{rgb}{0.77, 0.7, 0.35}
\definecolor{darkgoldenrod}{rgb}{0.72, 0.53, 0.04}
\definecolor{gold(metallic)}{rgb}{0.83, 0.69, 0.22}
\hypersetup{
 colorlinks=true,
 linkcolor=darkgoldenrod,
 filecolor=brown,      
 urlcolor=gold(metallic),
 citecolor=darkgoldenrod,
 pdftitle={On the mu equals zero conjecture for the fine Selmer group},
 }

\usepackage[all,cmtip]{xy}

\DeclareFontFamily{U}{wncy}{}
\DeclareFontShape{U}{wncy}{m}{n}{<->wncyr10}{}
\DeclareSymbolFont{mcy}{U}{wncy}{m}{n}
\DeclareMathSymbol{\Sh}{\mathord}{mcy}{"58}
\usepackage[T2A,T1]{fontenc}
\usepackage[OT2,T1]{fontenc}

\newtheorem{theorem}{Theorem}[section]
\newtheorem{lemma}[theorem]{Lemma}

\newtheorem{conjecture}[theorem]{Conjecture}

\newtheorem{proposition}[theorem]{Proposition}
\newtheorem{corollary}[theorem]{Corollary}
\newtheorem{definition}[theorem]{Definition}

\numberwithin{equation}{section}

\newcommand{\mufn}[1]{\mu^{\op{fn}}(#1)}

\theoremstyle{remark}
\newtheorem{remark}[theorem]{Remark}
\newtheorem{example}[theorem]{Example}

\newcommand{\Rfine}[2]{\mathcal{R}_{p^\infty}(#1/#2)}

\newcommand{\Gal}{\operatorname{Gal}}

\newcommand{\Zp}{\mathbb{Z}_p}

\newcommand{\FF}{\mathbb{F}}

\newcommand{\GL}{\mathrm{GL}}

\newcommand{\Z}{\mathbb{Z}}

\newcommand{\Q}{\mathbb{Q}}
\newcommand{\F}{\mathbb{F}}

\newcommand{\cO}{\mathcal{O}}

\newcommand{\Tst}{\mathbf{T}_{\rho^*}}

\newcommand{\cyc}{\mathrm{cyc}}

\newcommand{\op}[1]{\operatorname{#1}}

\newcommand\mtx[4] { \left( {\begin{array}{cc}
 #1 & #2 \\
 #3 & #4 \\
 \end{array} } \right)}

\begin{document}
\title[$\mu=0$ conjecture for Fine Selmer groups]{On the {\Large{$\mu$}} equals zero Conjecture for Fine Selmer groups in Iwasawa theory}

\author[Deo]{Shaunak V.~Deo}
\address[Deo]{Department of Mathematics, Indian Institute of Science, Bangalore 560012, India}
\email{shaunakdeo@iisc.ac.in}

\author[Ray]{Anwesh Ray}
\address[Ray]{Centre de recherches mathématiques,
Université de Montréal,
Pavillon André-Aisenstadt,
2920 Chemin de la tour,
Montréal (Québec) H3T 1J4, Canada}
\email{anwesh.ray@umontreal.ca}

\author[Sujatha]{R.~Sujatha}
\address[Sujatha]{Department of Mathematics\\
University of British Columbia\\
Vancouver BC, Canada V6T 1Z2}
\email{sujatha@math.ubc.ca}

\begin{abstract}
We study the Iwasawa theory of the fine Selmer group associated to Galois representations arising from modular forms. The vanishing of the $\mu$-invariant is shown to follow in some cases from a natural property satisfied by Galois deformation rings. We outline conditions under which the $\mu=0$ conjecture is shown to hold for various Galois representations of interest.
\end{abstract}


\maketitle
 
\section{Introduction}
\label{section:intro}
\par Let $F$ be a number field and $p$ be an odd prime. The $p$-primary roots of unity, considered
as a module over the absolute Galois group of $F$, is denoted by $\mu_{p^{\infty}}.$ Set $F_n$ to be the unique field
contained in $F(\mu_{p^{\infty}})$ such that $[F_n:F]=p^n$. 
The cyclotomic $\Z_p$-extension $F_{cyc}$ of $F$ is the subextension of $F(\mu_{p^{\infty}})$ obtained
as the union of the fields $F_n.$
In his seminal work \cite{iwasawa1973zl}, Iwasawa studied the growth of the $p$-part of the class groups over certain towers of number fields. In particular, for the cyclotomic $\Zp$-extension of $F$, Iwasawa proved that there are invariants $\mu, \lambda\in \Z_{\geq 0}$ and $\nu\in \Z$ such that \[\#\text{Cl}_p(F_n)=p^{p^n \mu + \lambda n + \nu} \text{ for } n \gg 0,\]
where $\text{Cl}_p(F_n)$ is the $p$-Sylow subgroup of the class group of $F_n$.
The invariants $\mu$ and $\lambda$ are the Iwasawa invariants associated to the $\Z_p$-extension $F_{\op{cyc}}/F$.

\par The genesis of Iwasawa theory arose from the study
of these invariants for a broad class of $\Zp$-extensions of a number field. 
Iwasawa formulated various conjectures about these invariants.
Amongst them is a famous conjecture,
henceforth referred to as {\it Iwasawa's $\mu=0$ conjecture}, which asserts that the $\mu$-invariant of the cyclotomic $\Zp$-extension $F_{\op{\cyc}}/F$ of any number field $F$ is zero.
Let $\mathcal{F}$ be the maximal abelian, unramified, pro-$p$ extension of $F_{\op{\cyc}}$, let $X(F_{\op{\cyc}}) = \op{Gal}(\mathcal{F}/F_{\op{\cyc}}) $ and let $\Gamma =\op{Gal}(F_{\op{\cyc}}/F)$.
Then an equivalent formulation of this conjecture is the statement that $X(F_{\op{\cyc}})$ is a finitely generated $\Z_p$-module.
When $F$ is an abelian number field, the $\mu$-invariant is known to vanish by the work of Ferrero and Washington \cite{ferrero1979iwasawa}. We remark that no general result in this direction is known for any broader class
of non-abelian number fields.

\par In the modern language of Galois representations and arithmetic geometry, Iwasawa's $\mu=0$ conjecture can be interpreted in terms of structural invariants associated to certain modules over the Iwasawa algebra of $\Gal(F_{\op{cyc}}/F)$.
 These modules arise naturally in the context of the Galois representation associated to the Tate motive $\mu_{p^{\infty}}.$ Subsequently, Iwasawa theory evolved to cover the study of a range of modules over Iwasawa algebras that arise from other Galois representations. An important instance of this evolution is the work of Mazur \cite{mazur1972rational} who initiated the Iwasawa theory of elliptic curves and abelian varieties. In this pathbreaking work, Mazur studied the $p$-primary Selmer group of an elliptic curve (or abelian variety) $E_{/F}$ which has good, ordinary reduction at all primes dividing $p$ over $F_{\op{cyc}}$, and defined analogous Iwasawa $\mu$ and $\lambda$-invariants in this context.
The $\mu$-invariant of the $p$-primary Selmer group need not vanish. For example, for $E_{/\Q}=X_0(11)$, the $\mu$-invariant of the $5$-primary Selmer group is positive. Given the analogy between class groups and Selmer groups of elliptic curves, it is natural to frame an analogue of Iwasawa's $\mu=0$ conjecture in this setting which we describe below.
\par Let $E$ be an elliptic curve defined over a number field $F$ and $p$ be a prime above which $E$ has potential good reduction. Set $\Gamma:=\op{Gal}(F_{\op{cyc}}/F)$ and let $\Lambda$ denote the corresponding Iwasawa algebra, which is defined to be the completed group algebra $\Lambda:=\varprojlim_n \Z_p[\Gamma/\Gamma^{p^n}]$. Then, the $p$-primary \emph{fine Selmer group} consists of the Galois cohomology classes in $H^1\left(\op{Gal}(\bar{F}/F_{\op{cyc}}), E[p^\infty]\right)$ that are trivial at all primes above $p$ and the primes at which $E$ has bad reduction. 
The fine Selmer group is believed to be cotorsion as an Iwasawa module and this is known to be true for elliptic curves over $\mathbb{Q}$ (see \cite{coates2005fine}).
In contrast, the usual Selmer group need not satisfy this condition unless the elliptic curve is assumed to have ordinary reduction at all primes above $p$. 

\par The properties of the fine Selmer group were systematically studied by Coates and the third named author who showed that the Iwasawa $\mu$-invariant of the fine Selmer group is related to the classical Iwasawa $\mu$-invariant of the cyclotomic $\Zp$-extension of the number field $F(E[p])$ cut out by the $p$-torsion in $E$. We refer to \cite[Corollary 3.5]{coates2005fine} for the precise statement of this result. This relationship shows that the vanishing of the $\mu$-invariant of the fine Selmer group is a consequence of Iwasawa's classical $\mu=0$ conjecture. This led Coates and the third named author to formulate the $\mu=0$ conjecture for the fine Selmer group of an elliptic curve (see \cite[Conjecture A]{coates2005fine}), as the analogue of Iwasawa's $\mu=0$ conjecture for Galois representations associated to elliptic curves.

 \par Iwasawa theory of  $p$-adic Galois representations is currently a central theme in number theory.
The fine Selmer group can be defined in the more general context of $p$-adic 
 Galois representations (see Definition~\ref{fineselmdef} and Definition~\ref{fineselmerdef}).
We remark that the $\mu=0$ conjecture can be formulated in this 
 broader framework (see Conjecture \ref{generalized mu equals 0 conjecture} and \cite[Conjecture A, p.74]{lim2018fine}). It has interesting applications to the study of Selmer groups associated to Galois representations of interest.
The goal of this article is to study the $\mu=0$ conjecture in this broader setting. The main results formulate explicit conditions under which the $\mu$-invariant vanishes for the fine Selmer groups associated to important classes of $p$-adic Galois representations. The cases of interest are the Galois representations arising from elliptic curves defined over number fields, Hecke eigencuspforms of weight at least $2$, Tate twists of Artin representations and Tate twists of the adjoint representations associated to modular Galois representations.

\par In the study of Galois representations, the theme of $p$-adic variation plays a key role. More precisely, the deformation theory of Galois representations, introduced and developed by Mazur in \cite{mazur1989deforming}, allows us to study $p$-adic families of lifts of a fixed mod-$p$ Galois representation. Hida \cite{hida2000adjoint} studied the relationship between certain modular deformation rings defined over the cyclotomic $\Z_p$-extension of $\Q$ and the Iwasawa theory of the classical Selmer group of the adjoint representation of a modular form. 
In this article, the $\mu=0$ conjecture for the fine Selmer group of the first Tate twist of the adjoint representation of a $p$-adic Galois representation $\rho$ is related to
\emph{unobstructedness} of the residual representation of $\rho$ (see Definition~\ref{def of unobstructed}).
Note that the notion of unobstructedness appears naturally in Galois deformation theory in connection with the structure of the universal deformation ring (in the sense of Mazur).
This provides a new approach in investigating the fine Selmer group. 
It is our belief that the linkage established here will help advance the current state of art in both the Iwasawa theory of the fine Selmer group as well as deformation theory of Galois representations. In particular, our methods provide a large class of new examples of fine Selmer groups of Galois representations for which the $\mu$-invariant vanishes.

\par Let us now briefly outline the results proved in this article. Section \ref{prelimsection} is preliminary in nature. In section \ref{criteriasection}, we establish a criterion for the vanishing of the $\mu$-invariant of the fine Selmer group of a Galois representation (see Theorem \ref{theorem 3.6}). 
To be precise, we prove that the $\mu$-invariant of the fine Selmer group of a $p$-adic Galois representation $\rho$ vanishes if the \emph{second} global Galois cohomology group of the first Tate twist of the dual of the residual representation of $\rho$ vanishes.
Thus, this criterion is determined purely in terms of the residual representation. Theorem \ref{th h2 vanishing} gives explicit conditions for the vanishing of the $\mu$-invariant of the fine Selmer group, which we specialize to the case of elliptic curves in Theorem \ref{elliptic curve criterion section 3}.

\par Section \ref{adjointsection} establishes the linkage between deformation theory of Galois representations and the vanishing of the $\mu$-invariant of the fine Selmer group. We study the fine Selmer group of the first Tate twist of the $n^2$-dimensional adjoint Galois representation $\op{Ad}\rho$ associated to an $n$-dimensional Galois representation $\rho$. An important class of such Galois representations arises naturally from automorphic forms. More precisely, Theorem \ref{unobstructed implies mu is zero} shows that the $\mu$-invariant vanishes for the fine Selmer group of $\op{Ad}\rho(1)$ if $\rho$ is \emph{unobstructed}. This concept was initially introduced by Mazur in \cite{mazur1989deforming}, and it was shown by Weston in \cite{weston2004unobstructed} that Galois representations arising from modular forms are unobstructed under some explicit conditions.

\par Let $f$ be a normalized Hecke newform of weight $k\geq 2$. Given a prime $\mathfrak{p}$ in the field of Fourier coefficients of $f$, let $\rho_{f,\mathfrak{p}}$ be the associated $2$-dimensional $\mathfrak{p}$-adic Galois representation. Theorem \ref{main theorem on adjoints} shows that the $\mu$-invariant of the fine Selmer group of $\op{Ad}\rho_{f, \mathfrak{p}}(1)$ vanishes for a set of primes $\mathfrak{p}$ of Dirichlet density $1$. 
Moreover, if $k>2$, then it is shown to vanish for all but finitely many primes $\mathfrak{p}$. Furthermore, if the level of $f$ is squarefree, then there is an explicit set of primes $\mathfrak{p}$ outside of which the associated $\mu$-invariant is $0$.
When $k=2$ and $f$ has trivial nebentypus, our results imply that the $\mu$-invariant of the fine Selmer group of $\op{Sym}^2 \rho_{f,\mathfrak{p}}$ vanishes for a set of primes $\mathfrak{p}$ of Dirichlet density $1$.
We also prove a similar result in the setting of Hilbert modular forms (see Theorem~\ref{hilbert-thm}).
On the other hand, we prove that if $E$ is a rational elliptic curve with squarefree conductor, then there are infinitely many primes $p$ such that the $\mu$-invariant of the fine Selmer group of $\op{Ad}^0\rho_{E,p}$ vanishes (see Theorem~\ref{ell thm}).
Here, $\rho_{E,p}$ is the $2$-dimensional $p$-adic Galois representation attached to $E$ and $\op{Ad}^0\rho_{E,p}$ is the subrepresentation of $\op{Ad}\rho_{E,p}$ consisting of matrices with trace $0$.

 We then turn our attention to Artin representations. Note that we can realize an Artin representation as a $p$-adic Galois representation for all primes $p$.
We prove that if $\rho$ is an Artin representation such that $H^0(\op{Gal}(\mathbb{C}/\mathbb{R}),\rho)=0$, then the $\mu$-invariants of the fine Selmer groups of the first Tate twists of the $p$-adic realizations of the dual of $\rho$ vanish for all but finitely many primes $p$ (see Theorem~\ref{artinthm}).

\par We also connect the $\mu=0$ conjecture with the notion of \emph{neatness} introduced by Mazur in \cite{mazur1989deforming} to prove that the $\mu=0$ conjecture is true for first Tate twist of the adjoint representation of a $p$-adic Galois representation $\rho$ whose residual representation is \emph{neat} (see Corollary \ref{corollary neat}). Leveraging a construction of Mazur and B\"{o}ckle, we end section \ref{adjointsection} with explicit families of odd and even $S_3$-representations such that the $\mu=0$-conjecture holds for the fine Selmer groups of first Tate twists of all their characteristic $0$ lifts. 

\par In section \ref{s 5}, we study the vanishing of the $\mu$-invariant of the fine Selmer group associated to $2$-dimensional residually dihedral Galois representations. The conditions are stated purely in terms of the residual representation and the set of primes at which the characteristic zero representation ramifies. In Theorem \ref{th 5.5} we give explicit conditions for the vanishing of the $\mu$-invariant and obtain a large class of examples for which our results apply. Finally, we remark that our results are illustrated with concrete numerical examples.

\subsection*{Acknowledgements} The first named author would like to thank Gabor Wiese for helping with computation of examples presented in \S\ref{s 5}.
The second named author is grateful to Ravi Ramakrishna for introducing him to the fascinating subject of Galois deformation theory, and would also like to thank Tom Weston and Jeffrey Hatley for helpful discussions along the way. His research is supported by the CRM-Simons bridge postdoctoral fellowship. The third named author gratefully acknowledges support from NSERC Discovery grant 2019-03987. The first and third named authors would also like to thank the online program on `Elliptic curves and the special values of $L$-functions' (code: ICTS/ECL2021/8) held at ICTS Bangalore, for providing the motivation and opportunity to begin discussions relevant to the work presented in this article.
 We would also like to thank the anonymous referee for providing many comments and suggestions which helped in improving the exposition.

\section{Preliminaries}
\label{prelimsection}
\par Throughout, $p$ is an odd prime number and $F$ is a number field. Set $S_p$ to be the set of primes of $F$ above $p$. Let $S$ be a finite set of primes of $F$ and assume that $S$ contains $S_p$ and all archimedean primes of $F$. Throughout, fix an algebraic closure $\bar{F}$ of $F$. Let $F_S$ be the maximal extension of $F$ contained in $\bar{F}$ in which all primes $v\notin S$ are unramified. Set $\op{G}_{F,S}$ to be the Galois group $\op{Gal}(F_S/F)$, identified with the maximal quotient of $\op{Gal}(\bar{F}/F)$ in which the primes $v\notin S$ are unramified. Given a field $L\subset F_S$, we set \[H^i\left(F_S/L, \cdot \right):=H^i\left(\op{Gal}(F_S/L), \cdot\right).\] For a number field $L$ contained in $F_S$, set
\[K^i_v(\cdot/L):=\bigoplus_{w|v} H^i(L_w, \cdot),\] where $w$ runs through the primes of $L$ above $v$. The fine Selmer group associated with a Galois module is obtained by imposing the \emph{strict local condition} at each prime that lies above the primes in $S$.
\begin{definition}
\label{fineselmdef}
Let $M$ be a discrete $p$-primary $\op{G}_{F,S}$ module and $L$ be a number field contained in $F_S$. The fine Selmer group over $L$ associated to $M$ is defined as follows
\[\Rfine{M}{L}:=\op{ker}\left\{H^1(F_S/L, M)\longrightarrow \bigoplus_{v\in S} K^1_v(M/L)\right\}.\]
Given an infinite extension $\mathcal{L}\subset F_S$, set 
\[\Rfine{M}{\mathcal{L}}:=\varinjlim_{L} \Rfine{M}{L},\] where $L$ runs over all number fields contained in $\mathcal{L}$.
\end{definition}
Note that, in general, the definition of $\Rfine{M}{\mathcal{L}}$ depends on the choice of $S$.
However, if $\mathcal{L}$ contains the cyclotomic $\Z_p$-extension of $F$, then $\Rfine{M}{\mathcal{L}}$ does not depend on the choice of $S$ (see \cite[Lemma 3.2]{SujathaWitte}).
\par For every integer $n\geq 1$, let $\mu_{p^n}$ be the group of $p^n$-th roots of unity contained in $\bar{F}$ and $\mu_{p^\infty}:=\cup_{n\geq 1} \mu_{p^n}$. Let $F(\mu_{p^n})$ (resp. $F(\mu_{p^\infty})$) be the field extension of $F$ generated by $\mu_{p^n}$ (resp. $\mu_{p^\infty}$). The cyclotomic $\Z_p$-extension of $F$ is the unique $\Z_p$-extension of $F$ contained in $F(\mu_{p^\infty})$, and is denoted by $F_{\op{cyc}}$.
Let $K$ be a finite extension of $\Q_p$ and $\cO$ be its valuation ring. Given an integral Galois representation, $\rho:\op{G}_{F,S}\rightarrow \op{GL}_n(\cO)$, let $\mathbf{T}_{\rho}$ be the underlying $\cO$-lattice on which $\op{G}_{F,S}$ acts by $\rho$ and set $\mathbf{A}(\rho):=\mathbf{T}_{\rho}\otimes_{\cO} \left(K/\cO\right)$. 
\begin{definition}
\label{fineselmerdef}
With respect to notation above, the fine Selmer group over $F_{\op{cyc}}$ associated with $\rho$ is $\Rfine{\mathbf{A}(\rho)}{F_{\op{cyc}}}$. 
\end{definition}
Thus \cite[Lemma 3.2]{SujathaWitte} implies that the definition of $\Rfine{\mathbf{A}(\rho)}{F_{\op{cyc}}}$ does not depend on the choice of $S$.
\par Set $\Gamma:=\op{Gal}(F_{\op{cyc}}/F)$ and let \[\Lambda:=\cO\llbracket \Gamma \rrbracket =\varprojlim_n \cO[\Gamma/\Gamma^{p^n}]\] be the associated Iwasawa algebra. Given a finitely generated $\Lambda$-module $\mathfrak{M}$, there is a map of $\Lambda$-modules\[
\mathfrak{M}\longrightarrow \Lambda^\alpha\oplus \left(\bigoplus_{i=1}^s \Lambda/(p^{\mu_i})\right)\oplus \left(\bigoplus_{j=1}^t \Lambda/(f_j(T)) \right)
\]
with finite kernel and cokernel. Here, $\alpha=\op{rank}_{\Lambda} \mathfrak{M}$ and each power series $f_j(T)$ is a \emph{distinguished polynomial}. In other words, $f_j(T)\in \Lambda$ is a monic polynomial all of whose non-leading coefficients are non-units in $\cO$. The $\mu$-invariant is the sum $\sum_i \mu_i$. Given an integral Galois representation
\[\rho:G_{F,S}\rightarrow \op{GL}_n(\cO),\]set $\mu^{\op{fn}}(\rho)=\mu^{\op{fn}}(\rho/F_{\op{cyc}})$ to be the $\mu$-invariant of $Y(\mathbf{A}(\rho)/F_{\op{cyc}})$, the Pontryagin dual of the fine Selmer group $\Rfine{\mathbf{A}(\rho)}{F_{\op{cyc}}}$. The following is Conjecture A in \cite{coates2005fine}.

\begin{conjecture}[Coates-Sujatha]\label{csconjecture}
Let $F$ be a number field, $E$ be an elliptic curve over $F$, 
and $S$ be a finite set of primes of $F$ containing $S_p$, all archimedean primes and all primes at which $E$ has bad reduction.
If \[\rho:G_{F,S}\rightarrow \op{GL}_2(\Z_p)\] is the $p$-adic Galois representation arising from $E$, then $\mu^{\op{fn}}(\rho)=0$.
\end{conjecture}
 The following result provides evidence for the above conjecture.
\begin{theorem}[Coates-Sujatha]
\label{coatessujathathm}
Let $F$ be a number field, $S$ be a finite set of primes of $F$ containing $S_p$ and all archimedean primes and \[\rho:G_{F,S}\rightarrow \op{GL}_2(\Z_p)\] be the $p$-adic Galois representation associated to an elliptic curve $E_{/F}$. Suppose that there is a number field $L/F$ such that 
\begin{enumerate}
    \item $L\subset F(E[p^\infty])$,
    \item the extension $F(E[p^\infty])/L$ is pro-$p$.
\end{enumerate}
Then, $\mufn{\rho}$ is equal to $0$ if and only if the Iwasawa's $\mu=0$ conjecture holds for the cyclotomic $\Zp$-extension $L_{\op{cyc}}/L$.
\end{theorem}
\begin{proof}
This result is \cite[Corollary 3.5]{coates2005fine}.
\end{proof}

\begin{theorem}\label{thm reducible mu equals 0}
Let $F$ be an abelian extension of $\Q$, and \[\rho:\op{Gal}(\bar{\Q}/\Q)\rightarrow \op{GL}_2(\Z_p)\] be the $p$-adic Galois representation associated to an elliptic curve $E_{/\Q}$. Assume that the residual representation $\bar{\rho}$ is reducible, i.e., there are characters \[\varphi_1, \varphi_2: \op{Gal}(\bar{\Q}/\Q)\rightarrow\F_p^\times\] such that $\bar{\rho}=\mtx{\varphi_1}{\ast}{0}{\varphi_2}$.
Let $S$ be a finite set of primes of $F$ containing $S_p$, all archimedean primes and all primes of bad reduction of $E_{/F}$.
Let $\rho_{F,S}$ be $\rho|_{G_{F,S}}$ i.e. $\rho_{F,S}$ is the representation $\rho$ when viewed as a representation $G_{F,S}$.
Then, $\mu^{\op{fn}}(\rho_{F,S}/F_{\op{cyc}})$ is equal to $0$.
\end{theorem}
\begin{proof}
Let $\Q(\varphi_i)$ be the extension of $\Q$ which is fixed by $\op{ker}\varphi_i$. Note that $\Q(\varphi_i)$ is an abelian extension of $\Q$. Let $L$ be the composite of the fields $F$, $\Q(\varphi_1)$ and $\Q(\varphi_2)$. Since it is a composite of abelian extensions of $\Q$, it follows that $L/\Q$ is abelian. By the result of Ferrero and Washington \cite{ferrero1979iwasawa}, the Iwasawa $\mu$-invariant vanishes for any abelian extension. Note that by construction, the extension $F\left(E[p^\infty]\right)/L$ is pro-$p$. Since the conditions of Theorem \ref{coatessujathathm} are met, the result follows. 
\end{proof}
\par We remark that the representation $\bar{\rho}$ in the statement of Theorem \ref{thm reducible mu equals 0} is reducible if $E$ admits a rational $p$-isogeny. In particular, if $E(\Q)[p]\neq 0$, then the residual representation is of the form $\bar{\rho}=\mtx{1}{\ast}{0}{\bar{\chi}_p}$, where $\bar{\chi}_p$ is the mod-$p$ cyclotomic character. Although the $\mu=0$ conjecture for fine Selmer groups was originally stated in \cite{coates2005fine} for the $2$-dimensional Galois representations associated to elliptic curves, the following generalized conjecture is expected to hold (cf. \cite[Conjecture A, p.74]{lim2018fine}):
\begin{conjecture}\label{generalized mu equals 0 conjecture}
Let $p$ be a prime number, $F$ be a number field and $\cO$ be the ring of integers of a finite extension of $\Q_p$. Let $S$ be a finite set of primes of $F$ containing the primes that lie above $p$ and $\rho:\op{G}_{F,S}\rightarrow \op{GL}_n(\cO)$ be an integral Galois representation. Then, $\mu^{\op{fn}}(\rho)$ is equal to $0$.
\end{conjecture}
\section{A Criterion for $\mu=0$ for the fine Selmer group of a Galois representation}
\label{criteriasection}
\par In this section, we shall discuss various criteria for the vanishing of the $\mu$-invariant of the fine Selmer group associated to a Galois representation. As in the previous section, let $p$ be an odd prime, $F$ be a number field and $S$ be a finite set of primes containing all archimedean primes and the primes above $p$. Let $\cO$ be the valuation ring of a finite extension $K/\Q_p$ with uniformizer $\varpi$ and let $\F:=\cO/\varpi$ be its residue field. Fix an integral Galois representation $\rho:\op{G}_{F,S}\rightarrow \op{GL}_n(\cO)$ on an $\cO$-lattice $\mathbf{T}_{\rho}$. Let $\bar{\rho}$ denote the residual representation $\rho\mod{\varpi}$ and set $\mathbf{V}_{\bar{\rho}}=\mathbf{T}_{\rho}\otimes_{\cO} \F$ to be the underlying vector space on which $\op{G}_{F,S}$ acts via $\bar{\rho}$. Setting $\mathbf{V}_{\rho}=\mathbf{T}_{\rho}\otimes_{\cO} K$, recall that $\mathbf{A}(\rho)=\mathbf{V}_{\rho}/\mathbf{T}_{\rho}$. We identify $\mathbf{V}_{\bar{\rho}}$ with $\mathbf{A}(\rho)[\varpi]$. Denote by $\rho^{\wedge}$ the dual of the representation $\rho$ and by $\chi_p$ be the $p$-adic cyclotomic character. We set $\rho^*$ to denote the representation $\rho^\wedge\otimes \chi_p$. Recall that $\bar{\chi}_p$ is the mod-$p$ cyclotomic character, and $\bar{\rho}^\wedge$ is the dual representation to $\bar{\rho}$. Identify $\bar{\rho}^*$ with $\bar{\rho}^\wedge \otimes \bar{\chi}_p$.
We note that $\mathbf{T}_{\rho^*}$ is isomorphic to the twist $\mathbf{T}_{{\rho}^\wedge}\otimes_{\cO} \cO(\chi_p)$.
Therefore, to summarize, $\rho$ and $\rho^*$ are integral representations acting on the $\cO$-modules $\mathbf{T}_\rho$ and $\mathbf{T}_{\rho^*}:=\op{Hom}\left(\mathbf{T}_{\rho}, \cO(\chi_p)\right)$, respectively. The representations $\bar{\rho}$ and $\bar{\rho}^*$ act on the $\F$-vector spaces $\mathbf{V}_{\bar{\rho}}$ and $\mathbf{V}_{\bar{\rho}^*}:=\op{Hom}\left(\mathbf{V}_{\bar{\rho}}, \F(\bar{\chi}_p)\right)$, respectively. We also note that $\mathbf{A}(\rho)$ is a $p$-divisible $K/\cO$-module which is identified with $\mathbf{T}_{\rho}\otimes_{\cO} K/\cO$.

\begin{conjecture}[Weak Leopoldt Conjecture for $\rho$]
With respect to notation above, $H^2(F_S/F_{\op{cyc}}, \mathbf{A}(\rho))=0$.
\end{conjecture}
We refer the reader to \cite[Appendix B]{perrinriou} for a list of cases where the Weak Leopoldt Conjecture is known.
\par We prove a criterion which shows that the vanishing of $\mufn{\rho}$ depends only on the residual representation $\bar{\rho}$. In the special case when $\rho$ is the Galois representation on the $p$-adic Tate module of an elliptic curve, a brief sketch of the proof is given in \cite[Proposition 4.6]{sujatha2010elliptic}. 
\begin{theorem}\label{Conj A criterion}
Let $\rho:\op{G}_{F,S}\rightarrow \op{GL}_n(\cO)$ be a Galois representation such that the weak Leopoldt conjecture holds for $\rho^*$. Then, the following are equivalent.
\begin{enumerate}
    \item\label{c1 Conj A criterion} The fine Selmer group $\Rfine{\mathbf{A}(\rho)}{F_{\op{cyc}}}$ is cotorsion over $\Lambda$ and $\mufn{\rho}=0$.
    \item\label{c2 Conj A criterion} The cohomology group $H^2(F_S/F_{\op{cyc}}, \mathbf{V}_{{\bar{\rho}}^*})$ is equal to $0$.
\end{enumerate}
\end{theorem}

\par The proof of the above theorem is provided later in this section, following Lemma \ref{torsionconditions}. The vanishing of $H^2(F_S/F_{\op{cyc}}, \mathbf{V}_{{\bar{\rho}}^*})$ is studied in greater detail in this article. Given a $\op{G}_{F,S}$-module $M$, the Iwasawa cohomology group $H^i_{\op{Iw}}\left(F_S/F_{\op{cyc}}, M\right)$ is defined to be the inverse limit $\varprojlim_n H^i\left(F_S/F_n, M\right)$ with respect to corestriction maps. Given a $\Lambda$-module $M$, we denote the Pontryagin dual by $M^{\vee}:=\op{Hom}_{\Z_p}\left(M, \Q_p/\Z_p\right)$. Recall that the fine Selmer group of $\mathbf{A}(\rho)$ over $F_{\op{cyc}}$ is $\Rfine{\mathbf{A}(\rho)}{F_{\op{cyc}}}$, and $Y(\mathbf{A}(\rho)/F_{\op{cyc}}):=\Rfine{\mathbf{A}(\rho)}{F_{\op{cyc}}}^{\vee}$. 
\begin{theorem}\label{ochivenjakobthm}
Let $\rho:\op{G}_{F,S}\rightarrow \op{GL}_n(\cO)$ be a Galois representation such that the weak Leopoldt conjecture holds for $\rho^*$. Then, $H^1(F_S/F_{\op{cyc}}, \mathbf{A}(\rho^*))^{\vee}$ contains no non-zero finite $\Lambda$-submodule (i.e., $\Lambda$-submodule having finite cardinality). 
\end{theorem}
\begin{proof}
The result is well known for Galois representations which are ordinary at $p$ (see \cite[Proposition 5]{Gr}) and the general case follows from \cite[Theorem 4.7]{ochi2002structure} specialized to the cyclotomic $\Z_p$-extension $F_{\op{cyc}}$.
\end{proof}
\begin{corollary}\label{finite if and only if 0}
Let $\rho:\op{G}_{F,S}\rightarrow \op{GL}_n(\cO)$ be a Galois representation and assume that the weak Leopoldt conjecture holds for $\rho^*$. Then, the following are equivalent:
\begin{enumerate}
    \item $H^2(F_S/F_{\op{cyc}}, \mathbf{V}_{\bar{\rho}^*})$ is finite,
    \item $H^2(F_S/F_{\op{cyc}}, \mathbf{V}_{\bar{\rho}^*})=0$.
\end{enumerate}
\end{corollary}
\begin{proof}
 For ease of notation, denote $\mathbf{A}(\rho^*)$ by $\mathbf{A}$ throughout the proof.
From the long exact sequence in cohomology associated with 
\[0\rightarrow \mathbf{A}[\varpi]\rightarrow \mathbf{A}\xrightarrow{\varpi} \mathbf{A}\rightarrow 0,\]
we have
\[ \dots \rightarrow H^2(F_S/F_{\op{cyc}}, \mathbf{A})^{\vee}\rightarrow H^2(F_S/F_{\op{cyc}}, \mathbf{A}[\varpi])^{\vee}\rightarrow H^1(F_S/F_{\op{cyc}}, \mathbf{A})^{\vee}.\]
From our hypothesis on Weak Leopoldt conjecture for $\rho^*$, we obtain that $ H^2(F_S/F_{\op{cyc}}, \mathbf{A}[\varpi])^{\vee}$ is a $\Lambda$-submodule of $H^1(F_S/F_{\op{cyc}}, \mathbf{A})^{\vee}$. By Theorem \ref{ochivenjakobthm}, $H^1(F_S/F_{\op{cyc}}, \mathbf{A})^{\vee}$ does not contain any non-zero finite $\Lambda$-submodule which proves the corollary.
\end{proof}
\par From the Poitou-Tate sequence associated to $\mathbf{A}(\rho)$, we obtain the following exact sequences (see \cite[Section 1.3.1]{perrinriou} for more details):
\begin{equation}\label{sesPT}\begin{split}
    0&\rightarrow H^0\left(F_{\op{cyc}}, \mathbf{A}(\rho)\right)\rightarrow \bigoplus_{v\in S} K_v^0(\mathbf{A}(\rho)/F_{\op{cyc}})\rightarrow H^2_{\op{Iw}}(F_S/F_{\op{cyc}},{\Tst})^{\vee} \rightarrow\\
    & R_{p^\infty}(\mathbf{A}(\rho)/F_{\op{cyc}})\rightarrow 0,\\
    0&\rightarrow R_{p^\infty}(\mathbf{A}(\rho)/F_{\op{cyc}})\rightarrow H^1(F_S/F_{\op{cyc}}, \mathbf{A}(\rho))\rightarrow \bigoplus_{v\in S} K^1_{v}(\mathbf{A}(\rho)/F_{\op{cyc}})\rightarrow\\
    & H_{\op{Iw}}^1(F_S/F_{\op{cyc}},{\Tst})^{\vee}\rightarrow H^2(F_S/F_{\op{cyc}}, \mathbf{A}(\rho)).
\end{split}\end{equation}

\begin{lemma}\label{torsionconditions}
 Let $\rho:\op{G}_{F,S}\rightarrow \op{GL}_n(\cO)$ be a Galois representation.
The following statements are equivalent
\begin{enumerate}
    \item\label{one} $Y(\mathbf{A}(\rho)/F_{\op{cyc}})$ is $\Lambda$-torsion with $\mu=0$,
    \item\label{two} $H^2_{\op{Iw}}\left(F_S/F_{\op{cyc}}, {\Tst}\right)$ is $\Lambda$-torsion with $\mu=0$.
\end{enumerate}

\end{lemma}

\begin{proof}
From the Poitou-Tate sequence, we get the following exact sequence
\[0\rightarrow Y(\mathbf{A}(\rho)/F_{\op{cyc}})\rightarrow H^2_{\op{Iw}}(F_S/F_{\op{cyc}},{\Tst})\rightarrow \bigoplus_{v\in S} K_v^0(\mathbf{A}(\rho)/F_{\op{cyc}})^{\vee}.\] 
Every non-archimedean prime $v$ is finitely decomposed in $F_{\op{cyc}}$, hence, it is not hard to see that $K_v^0(\mathbf{A}(\rho)/F_{\op{cyc}})^{\vee}$ is finitely generated as a $\Z_p$-module. As a $\Lambda$-module, $K_v^0(\mathbf{A}(\rho)/F_{\op{cyc}})^{\vee}$ is thus clearly torsion with $\mu$-invariant equal to $0$. The result is seen to follow from this.
\end{proof}

We now give a proof of Theorem \ref{Conj A criterion}. Denote the mod-$p$ Iwasawa algebra by $\Omega:=\Lambda/\varpi\Lambda$.

\begin{proof}[Proof of Theorem \ref{Conj A criterion}]
 For ease of notation, denote $\mathbf{A}({\rho^*})$ by $\mathbf{A}$ throughout the proof. {Identify $\mathbf{A}[\varpi]$ with $\mathbf{V}_{\bar{\rho}^*}$.} Note that $H_{\op{Iw}}^2(F_S/F_{\op{cyc}},{\Tst})$ is finitely generated as a $\Lambda$-module. This follows easily from the Poitou-Tate sequence \eqref{sesPT} together with the proof of Lemma~\ref{torsionconditions}. 
According to the Lemma \ref{torsionconditions}, condition \eqref{c1 Conj A criterion} is equivalent to $H^2_{\op{Iw}}\left(F_S/F_{\op{cyc}}, {\Tst}\right)$ being a finitely generated $\cO$-module. From the short exact sequence
\[{0\rightarrow \Tst\xrightarrow{\varpi} \Tst\rightarrow \mathbf{A}[\varpi]\rightarrow 0},\] we obtain
\[H_{\op{Iw}}^2(F_S/F_{\op{cyc}},{\Tst})\xrightarrow{\varpi}H_{\op{Iw}}^2(F_S/F_{\op{cyc}},{\Tst})\rightarrow H_{\op{Iw}}^2(F_S/F_{\op{cyc}},\mathbf{A}[\varpi])\rightarrow 0.\] The isomorphism
\[H_{\op{Iw}}^2(F_S/F_{\op{cyc}},\mathbf{A}[\varpi])\simeq \frac{H_{\op{Iw}}^2(F_S/F_{\op{cyc}},{\Tst})}{\varpi H_{\op{Iw}}^2(F_S/F_{\op{cyc}},{\Tst})}\]shows that \eqref{c1 Conj A criterion} is equivalent to the finiteness of $H_{\op{Iw}}^2(F_S/F_{\op{cyc}},\mathbf{A}[\varpi])$.

\par We show that \eqref{c2 Conj A criterion} is also equivalent to the finiteness of $H_{\op{Iw}}^2(F_S/F_{\op{cyc}},\mathbf{A}[\varpi])$. Given a finitely generated $\Omega$-module $M$, we set $E^j(M)$ to denote the Ext-group $\op{Ext}^j_{\Omega}(M, \Omega)$. Setting $E^{i,j}:=E^j\left(H^i(F_S/F_{\op{cyc}},\mathbf{A}[\varpi])^{\vee}\right)$, we find that the Iwasawa cohomology group $H_{\op{Iw}}^{2}(F_S/F_{\op{cyc}},\mathbf{A}[\varpi])$ is related to the cohomology groups $H^i\left(F_S/F_{\op{cyc}},\mathbf{A}[\varpi]\right)^{\vee}$ via Jannsen's spectral sequence \cite[Corollary 13]{jannsen2013spectral}. The spectral sequence states that
\[E^{i,j}\Rightarrow H_{\op{Iw}}^{i+j}(F_S/F_{\op{cyc}},\mathbf{A}[\varpi]).\]
Letting $W_i$ to be the cohomology group $H^i(F_S/F_{\op{cyc}}, \mathbf{A}[\varpi])^{\vee}$, note that for $j>0$, the adjoint Iwasawa module $E^j(W_i)$ is $\Omega$-torsion. Recall that condition \eqref{c2 Conj A criterion} is the requirement that $W_2=0$. According to Corollary \ref{finite if and only if 0}, $W_2=0$ if and only if $W_2$ is finite. On the other hand, it is clear that $W_2$ is finite if and only if $E^{2,0}=\op{Hom}_{\Omega}(W_2, \Omega)$ is $0$. Therefore, $W_2=0$ (i.e., \eqref{c2 Conj A criterion} holds) if and only if $E^{2,0}=0$. From the above spectral sequence, it follows that \eqref{c2 Conj A criterion} is equivalent to the finiteness of $H^2_{\op{Iw}}(F_S/F_{\op{cyc}}, \mathbf{A}[\varpi])$. 
\end{proof}
\begin{lemma}\label{lemma 3.3}
 Let $\rho:\op{G}_{F,S}\rightarrow \op{GL}_n(\cO)$ be a Galois representation.
{With notation as above, suppose that $H^2(F_S/F_{\op{cyc}}, \mathbf{V}_{\bar{\rho}^*})=0$. Then, the weak Leopoldt conjecture is true for $\rho^*$ and $\Rfine{\mathbf{A}(\rho)}{F_{\op{cyc}}}$ is cotorsion over $\Lambda$ with $\mufn{\rho}=0$.}
\end{lemma}
\begin{proof}
{Consider the short exact sequence of Galois modules 
\[0\rightarrow \mathbf{V}_{\bar{\rho}^*}\rightarrow \mathbf{A}(\rho^*)\xrightarrow{\varpi} \mathbf{A}(\rho^*)\rightarrow 0,\]
from which we find that $H^2(F_S/F_{\op{cyc}}, \mathbf{V}_{\bar{\rho}^*})$ surjects onto $H^2(F_S/F_{\op{cyc}}, \mathbf{A}(\rho^*))[\varpi]$. 
From the assumption $H^2(F_S/F_{\op{cyc}}, \mathbf{V}_{\bar{\rho}^*})=0$, it follows that $H^2(F_S/F_{\op{cyc}}, \mathbf{A}(\rho^*))[\varpi]=0$. 
The vanishing of $H^2(F_S/F_{\op{cyc}}, \mathbf{A}(\rho^*))[\varpi]$ implies that $H^2(F_S/F_{\op{cyc}}, \mathbf{A}(\rho^*))=0$. Thus, it follows that the weak Leopoldt conjecture holds for $\rho^*$. It follows from Theorem \ref{Conj A criterion} that $\Rfine{\mathbf{A}(\rho)}{F_{\op{cyc}}}$ is cotorsion over $\Lambda$ with $\mufn{\rho}=0$.}
\end{proof}

The following criterion will be used in establishing new cases of the $\mu=0$ conjecture for the fine Selmer group.

\begin{theorem}\label{theorem 3.6}
{Let $\rho : G_{F,S} \to \op{GL}_n(\cO)$ be a Galois representation as above. Suppose that $H^2(F_S/F, \mathbf{V}_{\bar{\rho}^*})=0$. Then, $\Rfine{\mathbf{A}(\rho)}{F_{\op{cyc}}}$ is cotorsion over $\Lambda$ with $\mufn{\rho}=0$. Furthermore, the weak Leopoldt conjecture holds for $\rho^*$.}
\end{theorem}
\begin{proof}
\par Since the $p$-cohomological dimension of $\Gamma = \Gal(F_{\op{\cyc}}/F)$ is $1$, by the Hochschild-Serre spectral sequence, the restriction map
\[H^2(F_S/F, \mathbf{V}_{\bar{\rho}^*})\rightarrow H^2(F_S/F_{\op{cyc}}, \mathbf{V}_{\bar{\rho}^*})^{\Gamma}\] is surjective. Hence, if $H^2(F_S/F, \mathbf{V}_{\bar{\rho}^*})=0$, then,
\[H^2(F_S/F_{\op{cyc}}, \mathbf{V}_{\bar{\rho}^*})^{\Gamma}=0.\] Since $\Gamma$ is pro-$p$, this implies that $H^2(F_S/F_{\op{cyc}}, \mathbf{V}_{\bar{\rho}^*})=0$, and the result follows from Lemma \ref{lemma 3.3}.
\end{proof}
\par In the remainder of this section, we outline sufficient conditions for the vanishing of {$H^2(F_S/F, \mathbf{V}_{\bar{\rho}^*})$}. We shall illustrate these conditions for Galois representations that arise from elliptic curves. In section \ref{s 5}, it is shown that these conditions are satisfied for a number of examples of \emph{dihedral Galois representations} of interest. Given $i\geq 0$, we let $\Sh^i_S(\mathbf{V}_{\bar{\rho}})$ be defined as follows
\[\Sh^i_S(\mathbf{V}_{\bar{\rho}}):=\op{ker}\left\{H^i(F_S/F, \mathbf{V}_{\bar{\rho}})\longrightarrow \bigoplus_{v\in S} H^i(F_v,\mathbf{V}_{\bar{\rho}}) \right\}.\]
Set $L$ to be the Galois extension of $F$ \emph{cut out} by $\bar{\rho}$. In other words, it is the extension of $F$ fixed by the kernel of $\bar{\rho}$. Set $H_L$ to be the mod-$p$ class group $H_L:=\op{Cl}(L)\otimes \F_p$. Denote by $H_L'$ the maximal quotient of $H_L$ such that the primes of $L$ that lie above $S$ are split in the corresponding subextension of the Hilbert class field of $L$. Note that $H_L'$ is a stable with respect to the natural action of $G=\op{Gal}(L/F)$ on $H_L$. Given a prime $v$ of $F$, set $\op{G}_v:=\op{Gal}(\bar{F}_v/F_v)$, which is viewed as a subgroup of $\op{Gal}(\bar{F}/F)$ after fixing an embedding $\bar{F} \to \bar{F_v}$. Set $\rho_{|v}$ (resp. $\bar{\rho}_{|v}$) to denote the restriction of $\rho$ (resp. $\bar{\rho}$) to $\op{G}_v$.

\begin{theorem}\label{th h2 vanishing}
Let $\rho:\op{G}_{F,S}\rightarrow \op{GL}_n(\cO)$ be a Galois representation and $\bar{\rho}:\op{G}_{F,S}\rightarrow \op{GL}_n(\F)$ be the associated residual representation.
With respect to notation above, assume that the following conditions are satisfied:
\begin{enumerate}
     \item\label{c3 of th h2} $H^1(G, \mathbf{V}_{\bar{\rho}})=0$,
    \item\label{c1 of th h2} $\op{Hom}_{G}(H_L', \mathbf{V}_{\bar{\rho}})=0$.
    \item\label{c2 of th h2} At each prime $v\in S$, the local representation $\bar{\rho}_{|v}$ does not have the trivial $1$-dimensional representation as a subrepresentation.
\end{enumerate}
Then,  $\Rfine{\mathbf{A}(\rho)}{F_{\op{cyc}}}$ is cotorsion over $\Lambda$ with $\mufn{\rho}=0$. Furthermore, the weak Leopoldt conjecture is true for $\rho^*$.
\end{theorem}
\begin{proof}
According to Theorem \ref{theorem 3.6}, if $H^2(F_S/F, \mathbf{V}_{\bar{\rho}^*})=0$, then the assertions follow. We show that $H^2(F_S/F, \mathbf{V}_{\bar{\rho}^*})=0$, thus proving the result. Note that $H^2(F_S/F,\mathbf{V}_{\bar{\rho}^*}) $ fits into an exact sequence
\[0\rightarrow \Sh^2_S(\mathbf{V}_{\bar{\rho}^*})\rightarrow H^2(F_S/F, \mathbf{V}_{\bar{\rho}^*})\rightarrow \bigoplus_{v\in S} H^2(F_v, \mathbf{V}_{\bar{\rho}^*}).\]
By Global duality of $\Sh$-groups \cite[Theorem 8.6.7]{NSW}, 
\[\Sh^2_S(\mathbf{V}_{\bar{\rho}^*})\simeq \Sh^1_S\left(\mathbf{V}_{\bar{\rho}}\right)^{\vee}.\] By local Tate duality \cite[pp. 91-92]{serre2013galois}, $H^2(F_v, \mathbf{V}_{\bar{\rho}^*})\simeq H^0(F_v,\mathbf{V}_{\bar{\rho}})^{\vee}$, and hence, is trivial since we are assuming that the trivial representation does not occur as a subrepresentation of $\bar{\rho}_{|v}$.
\par The hypothesis $H^1(G, \mathbf{V}_{\bar{\rho}})=0$, together with the inflation-restriction sequence applied to the extensions $F\subset L\subset F_S$,  implies that $\Sh^1_S\left(\mathbf{V}_{\bar{\rho}}\right)$ injects into $\op{Hom}_G(H_L', \mathbf{V}_{\bar{\rho}})$. Since it is assumed that $\op{Hom}_G(H_L', \mathbf{V}_{\bar{\rho}})=0$, it follows that $\Sh^1_S\left(\mathbf{V}_{\bar{\rho}}\right)=0$, and hence $\Sh^2_S(\mathbf{V}_{\bar{\rho}^*})=0$ as well. Thus, putting everything together, we have shown that $H^2(F_S/F, \mathbf{V}_{\bar{\rho}^*})$ $=0$. 
\end{proof}
\par Let $E$ be an elliptic curve defined over a number field $F$ and let $p$ be an odd prime number. Let $\rho:\op{G}_{F,S}\rightarrow \op{GL}_2(\Z_p)$ be the Galois representation on the $p$-adic Tate module of $E$ and let $\bar\rho$ be its residual representation. Let $G$ denote the Galois group $\op{Gal}(L/F)$.
For a prime $v$ of $F$, denote by $F_v$ the completion of $F$ at $v$ and by $\F_v$ the residue field of $F_v$.

\begin{theorem}\label{elliptic curve criterion section 3}
 Let $E_{/F}$ be an elliptic curve. 
 Let $S$ be the set of primes of $F$ that consists of primes $S_p$, all archimedean primes and the primes at which $E$ has bad reduction. Let $\rho : G_{F,S} \rightarrow \op{GL}_2(\Z_p)$ be the $p$-adic Galois representation attached to $E$. With respect to above notation, assume that the following conditions are satisfied
 \begin{enumerate}
    \item \label{c1 th 3.9}the residual representation $\bar{\rho}$ is irreducible,
     \item\label{c2 th 3.9} $\op{Hom}_G\left(H_L', E[p]\right)=0$,
     \item \label{c3 th 3.9}for every prime $v\in S$, we assume that $E(\F_v)[p]=0$.
 \end{enumerate}
 Then, the following assertions hold:
\begin{enumerate}[(a)]
    \item the weak Leopoldt conjecture is true for $\rho$,
    \item $\Rfine{E[p^\infty]}{F_{\op{cyc}}}$ is cotorsion over $\Lambda$ with $\mufn{\rho}=0$.
\end{enumerate}
\end{theorem}

\begin{proof}
We verify that the conditions of Theorem \ref{th h2 vanishing}. Clearly, the conditions \eqref{c1 of th h2} and \eqref{c2 of th h2} are satisfied. Note that $\mathbf{V}_{\bar{\rho}}$ is identified with $E[p]$. That condition \eqref{c3 of th h2} is satisfied follows from \cite[Lemma 2.2]{prasad2021relating}. This requires the assumption that $\bar{\rho}$ is irreducible.
 Hence, Theorem \ref{th h2 vanishing} implies that the weak Leopoldt conjecture is true for $\rho^*$ and $\Rfine{E[p^\infty]}{F_{\op{cyc}}}$ is cotorsion over $\Lambda$ with $\mufn{\rho}=0$.
As $\rho$ is the $p$-adic Galois representation attached to an elliptic curve, $\rho \simeq \rho^*$.
This finishes the proof of the theorem.
\end{proof}
We note here that for elliptic curves $E_{/\Q}$ with Mordell Weil rank $0$ and the Tate-Shafarevich group $\Sh(E/\Q)[p]=0$, the upper bound given by \cite[Theorem 4.2]{prasad2021relating} provides sufficient conditions for the vanishing of $\op{Hom}_G\left(H_L, E[p]\right)$. To illustrate this, consider the following example.

\begin{example}\label{example 1}
Let $E$ be the elliptic curve over $F=\Q$ given by \href{https://www.lmfdb.org/EllipticCurve/Q/11/a/1}{11a2} and Weierstrass equation $y^2+y=x^3-x^2-7820x-263580$ and set $p=7$. According to the data provided by the L-functions and Modular forms database \cite{cremona2021functions}, 
\begin{enumerate}
    \item $\op{rank} E(\Q)=0$,
    \item $E$ has conductor $11$ and thus, has good reduction at $7$,
    \item $E$ has split multiplicative reduction at $11$, thus, $\#E(\F_{11})=10$,
    \item $\#E(\F_7)=8-a_7(E)=10$,
    \item the mod-$7$ representation is irreducible,
    \item the Tamagawa product is $1$, 
    \item $\Sh(E/\Q)=0$. 
\end{enumerate}

Set $S=\{7,11,\infty\}$. The conditions \eqref{c1 th 3.9} and \eqref{c3 th 3.9} are satisfied. In order to see that condition \eqref{c2 th 3.9} is satisfied, it suffices to show that the stronger vanishing condition $\op{Hom}_G\left(H_L, E[p]\right)=0$. We refer to the notation in \cite[Theorem 4.2]{prasad2021relating}. Since the Mordell-Weil rank of $E$ is $0$ and $\Sh(E/\Q)=0$, it follows that the mod-$p$ Selmer group defined in \emph{loc.cit.} is $0$. Furthermore, since $7\nmid \#E(\F_{11})$ the set $\mathcal{I}$ in \emph{loc.cit.} is empty. Thus, by the upper bound of \cite[Theorem 4.2]{prasad2021relating}, we find that $\op{Hom}_G\left(H_L, E[p]\right)=0$, and in particular, \eqref{c2 th 3.9} is also satisfied for $E_{/\Q}$ at $p=7$.
\end{example}
\par We note here that our condition $\op{Hom}_G\left(H_L', E[p]\right)=0$ is a weaker condition than $\op{Hom}_G\left(H_L, E[p]\right)=0$. Since explicit computation with class groups of fields generated by torsion points of elliptic curves is difficult, we do not provide an example to illustrate this.

\section{The Adjoint representation of a modular form and $\mu=0$ for the fine Selmer group}
\label{adjointsection}

\par In this section, we prove results for the vanishing of the $\mu$-invariant of the fine Selmer group for the {first Tate-twist of the} \emph{adjoint Galois representation} associated with a newform. 
Here, by newform, we mean a new cuspidal eigenform. We will use this convention throughout the article.
When the newform has weight $2$ and trivial nebentypus, {the \emph{symmetric square representation} is a direct summand of this Tate-twist, and hence, in this case, we effectively obtain results for the symmetric square representation as well.} 
Moreover, given a rational elliptic curve $E$ of squarefree conductor, we prove results for the vanishing of the $\mu$-invariant of the fine Selmer group for $\op{Ad}^0(\rho_{E,p})$.
 Here $\rho_{E,p}$ is the $p$-adic Galois representation attached to $E$ and $\op{Ad}^0(\rho_{E,p})$ is the subrepresentation of the adjoint representation of $\rho_{E,p}$ consisting of matrices with trace $0$.
 
The study of these adjoint Galois representations is of special significance in deformation theory. We shall introduce ideas from deformation theory to study the vanishing of the $\mu$-invariant of the fine Selmer group.
Our methods also extend to Artin representations and \emph{neat} representations (in the sense of Mazur) and we prove results for the vanishing of the $\mu$-invariants of the fine Selmer groups of their first Tate-twists.

For a number field $K$, we will be studying continuous mod $p$ representations of $\op{Gal}(\overline{K}/K)$ which are unramified outside a finite set of primes of $K$ and their lifts to characteristic $0$.
Our focus will be on $\mu$-invariants of fine Selmer groups of such representations.
Of particular significance are the Galois representations arising from modular eigenforms.

\subsection{Adjoints of Galois representations}
\par The notion of \emph{unobstructedness} arises naturally in the study of Galois deformation theory. It was initially introduced by Mazur in \cite{mazur1989deforming} and is recalled below.
Let $\mathbb{F}$ be a finite field of characteristic $p$, $K$ be a number field and $S$ be a finite set of primes of $K$ containing all primes dividing $p$ and all archimedean primes.
Let $\bar\rho : G_{K,S} \to \op{GL}_n(\F)$ be a continuous representation.
Recall, from \S~\ref{prelimsection}, that $K_S$ is the maximal extension of $K$ unramified at all primes of $K$ lying outside $S$ and $\op{G}_{K,S}$ is the corresponding Galois group $\op{Gal}(K_S/F)$.

Let $\mathcal{O}$ be the ring of integers of a finite extension of $\Q_p$ with residue field $\F$.
Let $\op{CNL}_{\cO}$ be the category of complete local noetherian $\cO$-algebras $R$ with residue field isomorphic to $\F$. Given $R\in \op{CNL}_{\cO}$, let $\mathfrak{m}$ be the maximal ideal of $R$ and fix a residual isomorphism of $\cO$-algebras $R/\mathfrak{m}\xrightarrow{\sim} \F$. 
As $R$ is noetherian, there exist $\alpha_1,\cdots,\alpha_m \in R$ such that $\mathfrak{m}=(\alpha_1,\cdots,\alpha_m
)$.
Since $R$ is a complete local $\cO$-algebra, there exists a morphism $\phi : \cO\llbracket X_1,\cdots,X_m \rrbracket \to R$ which sends $X_i$ to $\alpha_i$ for all $1 \leq i \leq m$ (see \cite[Theorem 7.16 (a)]{eisenbud}).
Note that the map $\cO \to R/\mathfrak{m}$ induced by $\phi$ is surjective and $\alpha_1,\cdots,\alpha_m$ generate $\mathfrak{m}$.
Therefore, it follows, from \cite[Theorem 7.16 (b)]{eisenbud}, that $\phi$ is surjective. Hence, we get a presentation for $R$ of the form
\[R\simeq \frac{\cO\llbracket X_1,\dots , X_m\rrbracket}{\left(f_1,\dots, f_n\right)}.\] 
An $R$-lift of $\bar{\rho}$ is a Galois representation $\rho_R:\op{G}_{K,S}\rightarrow \op{GL}_n(R)$ such that $\bar{\rho}=\rho_R\mod{\mathfrak{m}}$. The subgroup of matrices $A\in \op{GL}_n(R)$ that reduce to the identity modulo the maximal ideal of $R$ is denoted by $\widehat{\op{GL}}_n(R)$. Two $R$-lifts $\rho_R$ and $\rho_R'$ of $\bar{\rho}$ are \emph{strictly equivalent} if $\rho_R=A \rho_R' A^{-1}$ for some matrix $A\in \widehat{\op{GL}}_n(R)$. An \emph{$R$-deformation} of $\bar{\rho}$ is a strict equivalence class of $R$-lifts.
\par Given a Galois representation $\rho_R:\op{G}_{K,S}\rightarrow \op{GL}_n(R)$, set $\op{Ad} \rho_R$ to be the adjoint representation of $\rho_R$. As an $R$-module, $\op{Ad}\rho_R$ consists of all $n\times n$ matrices with entries in $R$. The Galois action on $\op{Ad}\rho_R$ is defined by the adjoint action 
\[\sigma\cdot v=\rho_R(\sigma) v \rho_R(\sigma)^{-1},\] where, $\sigma\in \op{G}_{K,S}$ and $v\in \op{Ad}\rho_R$. Thus, we may view $\op{Ad}\rho_R$ as a Galois representation
\[\op{Ad}\rho_R:\op{G}_{K,S}\rightarrow \op{GL}_{n^2}(R)
.\]
In particular, specializing this construction to the case $R=\F$ and $\rho_R=\bar\rho$, we get the adjoint representation of $\bar\rho$:
\[\op{Ad}\bar\rho:\op{G}_{K,S}\rightarrow \op{GL}_{n^2}(\F)
.\]
\begin{definition}\label{def of unobstructed}
 A representation $\bar\rho : \op{G}_{K,S}\rightarrow \op{GL}_n(\F)$ is said to be \emph{unobstructed} if $H^2(K_S/K, \op{Ad}\bar{\rho})=0$.
Following this definition, we say that $\rho : \op{G}_{K,S}\rightarrow \op{GL}_n(\cO)$ is \emph{unobstructed} if the corresponding residual representation $\bar\rho : \op{G}_{K,S}\rightarrow \op{GL}_n(\F)$ (obtained by composing $\rho$ with the natural surjective map $\cO \to \F$) is unobstructed.
\end{definition}

Associated to $\bar{\rho} : G_{K,S} \to \op{GL}_n(\F)$, consider a \emph{functor of deformations}
\[\op{Def}_{\bar{\rho}}:\op{CNL}_{\cO}\rightarrow \op{Sets}, \] which takes $R\in \op{CNL}_{\cO}$ to the set $\op{Def}_{\bar{\rho}}(R)$ consisting of $R$-deformations of $\bar{\rho}$ considered as a representation of the group $\op{G}_{K,S}$. 
 When $\bar{\rho}$ is absolutely irreducible, there is a \emph{universal Galois representation} unramified away from $S$ (see \cite[Proposition 1]{mazur1989deforming}), which we denote by 
\[\rho^{\op{univ}}:\op{G}_{K,S}\rightarrow \op{GL}_n\left(R_{\bar{\rho}}\right).\]Here, $R_{\bar{\rho}}$ is the \emph{universal deformation ring} associated to $\bar{\rho}$ and the group $\op{G}_{K,S}$ (the set $S$ being suppressed in our notation). It is universal in the sense that for any $R$-deformation $\rho_R:\op{G}_{K,S}\rightarrow \op{GL}_n(R)$ of $\bar{\rho}$, there exists a unique homomorphism of complete noetherian local $\cO$-algebras $\psi:R_{\bar{\rho}}\rightarrow R$ such that the following diagram commutes
\[\begin{tikzcd}[column sep= large]
 & \text{GL}_n(R_{\bar{\rho}}) \arrow[d]\\
\op{G}_{K,S} \arrow[ru,"\rho^{\op{univ}}"] \arrow[r, "\rho_R"] & \text{GL}_n(R),\end{tikzcd}\]
where the vertical map is induced by $\psi$.  

In general, finding the explicit structure of the universal deformation ring $R_{\bar\rho}$ is a very difficult problem.
However, if $\bar\rho$ is unobstructed, then we know that 
\[R_{\bar\rho} \simeq \cO\llbracket X_1,\cdots,X_d\rrbracket\]
where $d = \dim_{\F}(H^1(G_{K,S},\op{Ad}\bar\rho))$ (see \cite[Proposition 2]{mazur1989deforming}).
Indeed, \cite[Section 1.6]{mazur1989deforming} implies that $R_{\bar\rho}$ is a quotient of $\cO\llbracket X_1,\cdots,X_d\rrbracket$. Since $\bar\rho$ is unobstructed, \cite[Proposition 2]{mazur1989deforming} implies that the Krull dimension of $R_{\bar\rho}$ is at least $d+1$. 
This allows us to conclude that $R_{\bar\rho}$ is isomorphic to $\cO\llbracket X_1,\cdots,X_d\rrbracket$.
In particular, specializing to the case $K=\Q$ and $n=2$, we get:
\begin{proposition}\label{power series unobstructed}
Suppose that $\bar\rho:\op{G}_{\Q,S}\rightarrow \op{GL}_2(\F)$ is absolutely irreducible and unobstructed. Then:
\begin{enumerate}
    \item If $\bar\rho$ is odd, then $R_{\bar{\rho}}$ is isomorphic (as a complete noetherian local $\cO$-algebra) to the formal power series ring $\cO\llbracket X_1,X_2, X_3\rrbracket$.
    \item If $\bar\rho$ is even, then $R_{\bar{\rho}}$ is isomorphic (as a complete noetherian local $\cO$-algebra) to the formal power series ring $\cO\llbracket X\rrbracket$.
\end{enumerate}
\end{proposition}
\begin{proof}
If $\bar\rho$ is odd, then the global Euler characteristic formula implies that $\dim_{\F}(H^1(G_{\Q,S},\op{Ad}\bar\rho))=3$.
If $\bar\rho$ is even, then the global Euler characteristic formula implies that $\dim_{\F}(H^1(G_{\Q,S},\op{Ad}\bar\rho))=1$.
The result now follows from \cite[Proposition 2]{mazur1989deforming} and the explanations given above.
\end{proof}

Let $\rho : G_{\Q,S} \to \op{GL}_2(\mathcal{O})$ be a representation. As we are assuming that $p$ is odd, we get a direct sum decomposition $$\op{Ad}\rho = \op{Ad}^0\rho \oplus \mathcal{O}$$ as $G_{\Q,S}$-representations with summands $\op{Ad}^0\rho$ and $\mathcal{O}$ corresponding to trace $0$ matrices and scalars, respectively. We have a similar decomposition for $\op{Ad}\bar\rho$ as well.
{Recall that we denoted the $p$-adic cyclotomic character by $\chi_p$.}
So we have $$\op{Ad}\rho(1)= \op{Ad}^0\rho(1) \oplus \mathcal{O}(\chi_p) =\op{Sym}^2(\rho) \otimes \det(\rho)^{-1}\chi_p \oplus \cO(\chi_p).$$
Thus, if $\det(\rho)= \chi_p$, then $\op{Sym}^2(\rho)$ is a subrepresentation of $\op{Ad}\rho(1)$.
We obtain the following criterion for the vanishing of the $\mu$-invariant of the fine Selmer group of the first Tate-twist of the adjoint representation (without assuming $\dim(\rho)=2$):

\begin{theorem}\label{unobstructed implies mu is zero}
Let $\rho:\op{G}_{K,S}\rightarrow \op{GL}_n(\cO)$ be a continuous Galois representation. If $\rho$ is \emph{unobstructed}, then the following assertions hold:
\begin{enumerate}
    \item\label{p1 of unobstructed implies mu is zero} the weak Leopoldt conjecture is true for $\op{Ad}\rho$.
    \item\label{p2 of unobstructed implies mu is zero} the fine Selmer group $\Rfine{\mathbf{A}(\op{Ad}\rho(1))}{K_{\op{cyc}}}$ associated to $\op{Ad}\rho(1)$ is cotorsion over $\Lambda$ with $\mufn{\op{Ad}\rho(1)}=0$.
    \item\label{p3 of unobstructed implies mu is zero} If $n=2$ and $r = \op{Sym}^2(\rho) \otimes (\det(\rho))^{-1}\chi_p$, then the fine Selmer group $\Rfine{\mathbf{A}(r)}{K_{\op{cyc}}}$ associated to $r$ is cotorsion over $\Lambda$ with $\mufn{r}=0$.
\end{enumerate}
\end{theorem} 

\begin{proof}
Observe that $\left(\op{Ad}\rho(1)\right)^* = \op{Ad}\rho$.
Unobstructedness of $\rho$ implies that $H^2(K_S/K, \op{Ad}\bar{\rho})=0$. Hence, the assertions \eqref{p1 of unobstructed implies mu is zero} and \eqref{p2 of unobstructed implies mu is zero} above follow directly from Theorem \ref{theorem 3.6}. For \eqref{p3 of unobstructed implies mu is zero}, consider the case when $n=2$ and $r=\op{Sym}^2(\rho) \otimes (\det(\rho))^{-1}\chi_p$. In this case, we find that $r^*=\op{Ad}^0\rho$. Hence, $H^2(K_S/K, \mathbf{V}_{\bar{r}^*})$ is a direct summand of $H^2(K_S/K, \op{Ad}\bar{\rho})$, and therefore is equal to $0$. The assertion in this case follows once again from Theorem \ref{theorem 3.6}.
\end{proof}

The unobstructedness of modular Galois representations $\rho$ is studied in greater detail in \cite{weston2004unobstructed}.
Before recalling the main result of \cite{weston2004unobstructed}, we briefly describe its setup.
Let $f$ be a newform (i.e. a new cuspidal eigenform) of weight $k\geq 2$ on $\Gamma_1(N)$, where $N\in \Z_{\geq 1}$. 
Let $S$ be a finite set of primes of $\Q$ containing the primes dividing $N$ and $\infty$.
Denote by $F$ the field of Fourier coefficients of $f$ and let $\mathfrak{p}$ be a prime of $\cO_F$. Set $\cO$ to be the completion of $\cO_F$ at $\mathfrak{p}$. Fix a uniformizer $\varpi$ of $\cO$ and set $\F$ to denote the residue field $\cO/\varpi$. Let $p$ be the prime of $\Q$ lying below $\mathfrak{p}$ and let $S_{\mathfrak{p}} := S \cup \{p\}$. Let
\[\rho_{f, \mathfrak{p}}: \op{G}_{\Q,S_{\mathfrak{p}}} \rightarrow \op{GL}_2(\cO)\] be an integral Galois representation associated to $f$ and the prime $\mathfrak{p}$ {by the construction of Eichler-Shimura and Deligne}. Let $\bar{\rho}_{f, \mathfrak{p}}:\op{G}_{\Q,S_{\mathfrak{p}}}\rightarrow \op{GL}_2(\F)$ be the residual representation obtained by reducing $\rho_{f,\mathfrak{p}}$ modulo $\varpi$. Note that the construction of Eichler-Shimura and Deligne gives a Galois representation over the fraction field of $\mathcal{O}$ and the Galois representation $\rho_{f,\mathfrak{p}}$ arises from a choice of a Galois stable $\cO$-lattice $\mathbf{T}$ in it. When the residual representation is absolutely irreducible, there is a unique choice of $\mathbf{T}$.
In all the examples that we consider, the residual representation will be absolutely irreducible. 
 We will mostly be working with this setup in the rest of the section.

Denote by $\op{Obs}(f)$ the set of primes $\mathfrak{p}$ of $\cO_F$ at which $\rho_{f, \mathfrak{p}}$ is \emph{obstructed} (i.e., not unobstructed). An important notion in this context is that of a \emph{congruence prime}.
\begin{definition}\label{congdef}
Let $d$ be a divisor of $N$. We say that $\mathfrak{p}$ is a \emph{congruence prime of level $d$} for $f$ if there is a newform $f'$ such that
\begin{itemize}
    \item $f'$ has weight $k$ and level $\Gamma_1(d)$;
    \item $f'$ is \emph{not} Galois conjugate to $f$;
    \item $\bar\rho_{f', \bar{\mathfrak{p}}}\simeq \bar\rho_{f, \bar{\mathfrak{p}}}$ for some prime $\bar{\mathfrak{p}}$ of $\overline{\Q}$ above $\mathfrak{p}$.
\end{itemize}
Denote by $\op{Cong}(f)$ the set of all congruence primes of $f$ (as $d$ ranges through all divisors of $N$).
\end{definition}
Since there are only finitely many newforms whose level is a divisor of $N$, it is easy to see, by comparing their Hecke eigenvalues which are not equal, that $\op{Cong}(f)$ is finite. Let $\varphi$ denote Euler's totient function.

\par The following explicit result is due to Weston and has been subsequently generalized in various directions.

\begin{theorem}[Weston]\label{weston theorem}
Let $f$ be a Hecke newform of weight $k\geq 2$ on $\Gamma_1(N)$ and let $M$ be the conductor of the nebentypus of $f$. Consider two cases.
\begin{enumerate}
    \item First, consider the case $k>2$. Then, for all but finitely many primes $\mathfrak{p}$ of $\cO_F$, the Galois representation $\rho_{f, \mathfrak{p}}$ is unobstructed. Furthermore, if $N$ is assumed to be squarefree, we have that
    \[\op{Obs}(f)\subseteq \left\{\mathfrak{p}|p; p\leq k+1\right\} \cup \left\{\mathfrak{p}|p; p|\left(N\varphi(N)\prod_{\ell|(N/M)} (\ell+1)\right)\right\}\cup \op{Cong}(f).\]
    \item Consider the case $k=2$. Then, the set $\op{Obs}(f)$ of obstructed primes $\mathfrak{p}$ has Dirichlet density zero.
\end{enumerate}
\end{theorem}
\begin{proof}
The first part is obtained by combining \cite[Theorem 5.4]{weston2004unobstructed} with \cite[Theorem 1]{weston2005explicit}.
The second part follows directly from \cite[Theorem 5.5]{weston2004unobstructed} (see also \cite[Corollary 2]{mazur1997fern} when $f$ is a non-CM modular newform having weight $2$, trivial nebentypus and rational Hecke eigenvalues).
\end{proof}
Combining the above theorem with Theorem \ref{unobstructed implies mu is zero}, one obtains an immediate application towards the vanishing of the $\mu$-invariant of the fine Selmer group. 
\begin{theorem}\label{main theorem on adjoints}
Let $f$ be a Hecke newform of weight $k\geq 2$ on $\Gamma_1(N)$ and let $M$ be the conductor of the nebentypus $\epsilon$ of $f$. Let $F$ be the number field generated by its Hecke eigenvalues and let $\cO_F$ be the ring of integers of $F$. Consider two cases.
\begin{enumerate}
    \item If $k>2$, then, for all but finitely many primes $\mathfrak{p}$ of $\cO_F$, the following assertions hold:
\begin{enumerate}
    \item the weak Leopoldt conjecture holds for $\op{Ad}\rho_{f,\mathfrak{p}}$.
    \item {Let $r$ denote any one of the representations {$\op{Sym}^2(\rho_{f,\mathfrak{p}}) \otimes (\det(\rho_{f,\mathfrak{p}}))^{-1} \chi_p$} or $\op{Ad}\rho_{f,\mathfrak{p}}(1)$. Then the fine Selmer group associated to $r$ is cotorsion over $\Lambda$ with $\mufn{r}=0$.}
\end{enumerate} Furthermore, if $N$ is squarefree, then the above assertions hold for all primes outside the finite set
    \[\op{Obs}(f)\subseteq \left\{\mathfrak{p}|p; p\leq k+1\right\} \cup \left\{\mathfrak{p}|p; p|\left(N\varphi(N)\prod_{\ell|(N/M)} (\ell+1)\right)\right\}\cup \op{Cong}(f).\]
    \item Consider the case $k=2$. Then, for a set of primes $\mathfrak{p}$ of $\cO_F$ of Dirichlet density one, the following assertions hold:
\begin{enumerate}
    \item the weak Leopoldt conjecture is true for $\op{Ad}\rho_{f,\mathfrak{p}}$.
    \item {For $r\in \{\op{Sym}^2(\rho_{f,\mathfrak{p}}) \otimes \epsilon^{-1},\op{Ad}\rho_{f,\mathfrak{p}}(1)\}$, the fine Selmer group associated to $r$ is cotorsion over $\Lambda$ with $\mufn{r}=0$.}
\end{enumerate}
\end{enumerate}
\end{theorem}

\begin{proof}
The result follows directly from Theorems \ref{unobstructed implies mu is zero} and \ref{weston theorem}.
\end{proof}


\begin{remark}
If $f$ is a newform of weight $2$ with trivial nebentypus, then, under the notation established above, Theorem~\ref{main theorem on adjoints} implies that for a set of primes $\mathfrak{p}$ of $\mathcal{O}_F$ of Dirichlet density one, the fine Selmer group associated to the symmetric square representation $\op{Sym}^2(\rho_{f,\mathfrak{p}})$ is co-torsion over $\Lambda$ with $\mufn{\op{Sym}^2(\rho_{f,\mathfrak{p}})}=0$.
\end{remark}

\begin{remark}
For Galois representations associated to elliptic curves of rank $0$ (not the {adjoint or symmetric square} representation, but the representation itself), similar results can be proved for a set of primes $\mathfrak{p}$ of Dirichlet density $1$, see \cite{wuthrich2007iwasawa}. 
On the other hand, for modular forms of weight $k>2$, Theorem~\ref{main theorem on adjoints} proves that the $\mu$-invariant of the first Tate twist of the adjoint representation vanishes for all but finitely many primes, which is stronger than showing that it vanishes for a set of primes of Dirichlet density $1$. Moreover, the set of primes outside which the $\mu$-invariant is known to vanish, is made explicit in the squarefree level case by the above result. Also, in the weight $2$ case, when $f$ coincides with an abelian variety of $\op{GL}_2$-type, there is no assumption made on the rank of this abelian variety.
\end{remark}

\begin{remark}
Hatley \cite{hatley2016obstruction} has generalized \cite[Theorem 1]{weston2005explicit} to modular newforms of arbitrary level.
Thus, combining \cite[Theorem 3.6]{hatley2016obstruction} with Theorem~\ref{theorem 3.6}, one can explicitly describe a finite set of primes of $\cO_F$ outside of which $\mufn{{\op{Sym}^2(\rho_{f,\mathfrak{p}}) \otimes (\det(\rho_{f,\mathfrak{p}}))^{-1}\chi_p}}$ and $\mufn{\op{Ad}\rho_{f,\mathfrak{p}}(1)}$ vanish even when the level of $f$ is not squarefree.
Since the description of this set is tedious, we will not give it here.
We refer the reader to \cite[Theorem 3.6]{hatley2016obstruction} for more details.
\end{remark}

Specializing to the case when $N=1$, $S =\{\infty\}$ and $f$ is the unique normalized cusp form of level $1$ and weight $k \in \{12,16,18,20,22,26\}$,
we get, for every prime $p$, the $p$-adic Galois representation $$\rho_{f,p} : G_{\Q,\{p,\infty\}} \to \op{GL}_2(\mathbb{Z}_{p})$$ attached to $f$.
Let $$\bar\rho_{f,p} : G_{\Q,\{p,\infty\}} \to \op{GL}_2(\mathbb{F}_{p})$$ be the reduction of $\rho$ modulo $p$. 

Combining \cite[Theorem 2]{weston2004unobstructed} with {Theorem~\ref{unobstructed implies mu is zero}}, we get:
\begin{theorem}
\label{level 1 theorem}
 Let $f$ be the unique normalized cusp form of level $1$ and weight $k \in \{12,16,18,20,22,26\}$.
 If $p > k+1$ is a prime such that $\bar\rho_{f,p}$ is absolutely irreducible, then
 \begin{enumerate}
    \item the weak Leopoldt conjecture is true for $\op{Ad}\rho_{f,p}$.
    \item {For $r\in \{\op{Sym}^2(\rho_{f,p})(2-k),\op{Ad}\rho_{f,p}(1)\}$, the fine Selmer group associated to $r$ is cotorsion over $\Lambda$ with $\mufn{r}=0$.}
\end{enumerate}
\end{theorem}
\begin{proof}
In the setting of the theorem above, Weston obtains \cite[Theorem 2]{weston2004unobstructed} by proving that $\rho_{f,p}$ is unobstructed i.e. $H^2(\mathbb{Q}_{\{p,\infty\}}/\mathbb{Q}, \op{Ad}(\bar\rho_{f,p}))=0$ (see the proof of \cite[Theorem 5.6]{weston2004unobstructed} for more details).
Combining this with Theorem~\ref{unobstructed implies mu is zero} proves the theorem.
Note that, in the cases at hand, $\det(\rho_{f,p}) =\chi_p^{k-1}$.
Hence, we have $\op{Sym}^2(\rho_{f,\mathfrak{p}}) \otimes (\det(\rho_{f,\mathfrak{p}}))^{-1} \chi_p = \op{Sym}^2(\rho_{f,p})(2-k)$.
\end{proof}
If $f$ is one of the newforms considered in Theorem~\ref{level 1 theorem}, then the set of primes $p$ such that $\bar\rho_{f,p}$ is \emph{not} absolutely irreducible is given in \cite[Section 5.4]{weston2004unobstructed}.
In particular if $f = \Delta$, then the hypotheses of Theorem~\ref{level 1 theorem} are satisfied by primes $p \geq 17$ and $p \neq 691$.

Weston's work on unobstructedness of modular Galois representations has been generalized in several directions by various authors. 
In the setting of Hilbert modular forms, such results have been obtained by Gamzon \cite{gamzon2016unobstructed}.
Using his main result, we get:
\begin{theorem}\label{hilbert-thm}
 Let $K$ be a totally real field, $f$ be a Hilbert modular newform over $K$ and $S$ be a finite set of primes of $K$ containing all prime divisors of the level of $f$ and all archimedean primes of $K$. 
 Let $F$ be the number field generated by its Hecke eigenvalues and let $\cO_F$ be the ring of integers of $F$.
 For a prime $\mathfrak{p}$ of $\cO_F$, let $p$ be the rational prime lying below $\mathfrak{p}$, $S_{\mathfrak{p}}$ be $S \cup \{\text{primes of } K \text{ lying above } p\}$ and $$\rho_{f,\mathfrak{p}}: G_{K,S_{\mathfrak{p}}} \to \op{GL}_2(\cO_{F,\mathfrak{p}})$$ be the corresponding Galois representation attached to $f$.
 Suppose the following hypotheses hold:
 \begin{enumerate}
     \item $f$ has no CM,
     \item $f$ is not a twist of a base change of a Hilbert newform over a proper subfield $E$ of $K$,
     \item All weights of $f$ are greater than $2$.
 \end{enumerate}
  Then, for all but finitely many primes $\mathfrak{p}$ of $\cO_F$, the following assertions hold:
\begin{enumerate}[(a)]
    \item the weak Leopoldt conjecture is true for $\op{Ad}^0\rho_{f,\mathfrak{p}}$.
    \item {The fine Selmer group associated to $\op{Sym}^2(\rho_{f,\mathfrak{p}}) \otimes (\det(\rho_{f,\mathfrak{p}}))^{-1} \chi_p$ is cotorsion over $\Lambda$ with $\mufn{\op{Sym}^2(\rho_{f,\mathfrak{p}}) \otimes (\det(\rho_{f,\mathfrak{p}}))^{-1} \chi_p}=0$}.
\end{enumerate}
\end{theorem}

\begin{proof}
Note that, under the hypotheses of the theorem above, \cite[Theorem 1.1]{gamzon2016unobstructed} implies that $H^2(K_{S_{\mathfrak{p}}}/K,\op{Ad}^0\bar{\rho}_{f,\mathfrak{p}}) =0$ for all but finitely many primes $\mathfrak{p}$ of $\cO_F$. {We note here that $\left(\op{Sym}^2(\rho_{f,\mathfrak{p}}) \otimes (\det(\rho_{f,\mathfrak{p}}))^{-1} \chi_p\right)^*$} is identified with $\op{Ad}^0\rho_{f,\mathfrak{p}}$.
Theorem~\ref{theorem 3.6} gives the result.
\end{proof}

The analogous problem for $\op{GSp}_4$-representations is studied by Broshi, Mullath, Sorensen and Weston \cite{broshi2020unobstructed}. 
 On the other hand, Guiraud in \cite{guiraud2020unobstructedness} has established a related generalization (in the spirit of Gamzon) in the setting of regular algebraic conjugate self-dual cuspidal (RACSDC) automorphic representations (see \cite[Theorem 1.2]{guiraud2020unobstructedness} for more details).
We expect that Theorem \ref{main theorem on adjoints} can be suitably generalized to such settings.

We end this subsection by proving an analogue of Theorem~\ref{main theorem on adjoints} for the vanishing of $\mu$-invariant of $\op{Ad}^0\rho$, where $\rho$ is the Galois representation associated to an elliptic curve.

\begin{theorem}\label{ell thm}
Let $E$ be an elliptic curve over $\Q$ with squarefree conductor $N$ and let $S$ be the set of primes dividing $N$ and $\infty$. For a prime $p$, let $$\rho_{E,p} : G_{\Q,S_p} \to \op{GL}_2(\Z_p)$$ be the $p$-adic Galois representation attached to $E$.
Then, for infinitely many primes $p$, the following assertions hold:
\begin{enumerate}
    \item the weak Leopoldt conjecture holds for $\op{Ad}^0\rho_{E,p}(1)= \op{Sym}^2(\rho_{E,p})$.
    \item The fine Selmer group associated to $\op{Ad}^0\rho_{E,p}$ is cotorsion over $\Lambda$ with $\mufn{\op{Ad}^0\rho_{E,p}}=0$.
\end{enumerate}
\end{theorem}

\begin{proof}
Recall that $(\op{Ad}^0\rho_{E,p})^* = \op{Ad}^0\rho_{E,p}(1)$.
So, by Theorem~\ref{theorem 3.6}, it suffices to prove that for infinitely many primes p, we have $H^2(\Q_{S_p}/\Q, \op{Ad}^0\bar{\rho}_{E,p}(1))=0$.
Using Poitou-Tate duality, we get the following exact sequence (see \cite[Proposition 10]{Washington}):
\begin{multline}0 \to H^0(\Q_{S_p}/\Q,\op{Ad}^0\bar\rho_{E,p}) \to \oplus_{\ell \in S} H^0(G_{\ell},\op{Ad}^0\bar\rho_{E,p}) \oplus H^0(G_p,\op{Ad}^0\bar\rho_{E,p}) \to \\ H^2(\Q_{S_p}/\Q, \op{Ad}^0\bar{\rho}_{E,p}(1))^* \to \Sh^1_{S_p}(\op{Ad}^0\bar\rho_{E,p}) \to 0.
\end{multline}
Hence, we conclude that $H^2(\Q_{S_p}/\Q, \op{Ad}^0\bar{\rho}_{E,p}(1))=0$ if the following conditions hold:
\begin{enumerate}
    \item $H^0(G_{\ell}, {\op{Ad}^0\bar{\rho}_{E,p}})=0$ for all $\ell \in S$,
    \item $H^0(G_p, {\op{Ad}^0\bar{\rho}_{E,p}})=0$,
    \item $\Sh^1_{S_p}(\op{Ad}^0\bar{\rho}_{E,p})=0$.
\end{enumerate} 

Suppose $\bar{\rho}_{E,p}$ is absolutely irreducible.
Hence, it follows that $H^0(\Q_{S_p}/\Q,\op{Ad}^0\bar{\rho}_{E,p})=0$.
Note that, in \cite{weston2005explicit}, Weston defines a suitable Selmer group $H^1_{\emptyset}(G_{\Q},\mathbf{A}(\op{Ad}^0\rho_{E,p}))$ of $\mathbf{A}(\op{Ad}^0\rho_{E,p})$ (see \cite[p. 204]{weston2005explicit}).
This Selmer group is also extensively studied in \cite{dfg}.
Since $H^0(\Q_{S_p}/\Q,\op{Ad}^0\bar{\rho}_{E,p})=0$, we get, using the proof of \cite[Lemma 7]{weston2005explicit} and \cite[Equation 2.3]{weston2005explicit}, that $\Sh^1_{S_p}(\op{Ad}^0\bar{\rho}_{E,p})$ is a subgroup of $H^1_{\emptyset}(G_{\Q},\mathbf{A}(\op{Ad}^0\rho_{E,p}))$.
Let $f_E$ be the weight $2$ modular form associated to $E$.
From \cite[Theorem 3.7]{dfg}, we get that the length of  $H^1_{\emptyset}(G_{\Q},\mathbf{A}(\op{Ad}^0\rho_{E,p}))$ is the $p$-valuation of the congruence ideal $\eta_{f_E}^{\emptyset}$ defined in \cite[Section 1.7.3]{dfg}. As $\eta_{f_E}^{\emptyset}$ is an ideal of $\mathbb{Z}$, we conclude that $H^1_{\emptyset}(G_{\Q},\mathbf{A}(\op{Ad}^0\rho_{E,p})) =0$ for all but finitely many primes $p$ (see the proofs of \cite[Theorem 5.4, Theorem 5.6]{weston2004unobstructed} for more details). 
Note that $\bar{\rho}_{E,p}$ is absolutely irreducible for all but finitely many primes $p$.
Hence, we conclude that $\Sh^1_{S_p}(\op{Ad}^0\bar{\rho}_{E,p})=0$ for all but finitely many primes $p$.

Now $E$ is a rational elliptic curve of conductor $N$. So it follows, from the modularity of rational elliptic curves, that there exists a newform $f_E$ of level $\Gamma_0(N)$ and weight $2$ with rational Fourier coefficients such that $\rho_{E,p} = \rho_{f_E,p}$ for all primes $p$.

Now, in addition to irreducibility of $\bar{\rho}_{E,p}$, suppose $p > N^2$ and $H^0(G_{\ell}, \op{Ad}^0\bar{\rho}_{E,p}) \neq 0$ for some $\ell \in S$.
As $N$ is squarefree and $f_E$ has trivial nebentypus, it follows that $\bar{\rho}_{E,p}$ is reducible, semi-simple and hence, unramified at $\ell$ (see the proof of \cite[Lemma 11]{weston2005explicit} for more details).
So, we conclude, from \cite[(B) of p. 221]{edixhoven}, that $p$ is a congruence prime of level dividing $\dfrac{N}{\ell}$ (see Definition~\ref{congdef} for the definition of congruence primes).
Since $\bar{\rho}_{E,p}$ is absolutely irreducible for all but finitely many primes $p$ and there are only finitely many congruence primes, we conclude that for all but finitely many primes $p$, $H^0(G_{\ell}, \op{Ad}^0\bar{\rho}_{E,p})=0$ for all $\ell \in S$.

Note that if $p$ is a supersingular prime of $E$, then the restriction of $\bar{\rho}_{E,p}$ to $G_p$ is absolutely irreducible and hence, $H^0(G_p, \op{Ad}^0\bar{\rho}_{E,p})=0$.
A celebrated theorem of Elkies (\cite[Theorem 2]{elkies}) implies that $E$ has infinitely many supersingular primes. 
Combining this with all the analysis given above proves the theorem.
\end{proof}

\subsection{Artin representations}
Note that, in Theorem~\ref{weston theorem}, the weight of the modular newforms is always assumed to be greater than $1$.
Therefore, it does not shed any light on the unobstructedness of Artin representations as they arise only in the setting of weight $1$ modular forms.
Moreover, an Artin representation can be considered as a $p$-adic representation for every prime $p$ (as we will see below). 
So it is easy to formulate the question studied by Weston \cite{weston2004unobstructed} for Artin representations of arbitrary dimensions.
This question is studied by B\"{o}ckle, Guiraud, Kalyanswamy and Khare in \cite{khare2018leopoldt} in more generality.
In particular, they focus on the vanishing of $H^2$ for arbitrary Artin representations rather than restricting to the case of adjoint of Artin representations.
We will now briefly describe their setup following \cite[Section 6.1]{khare2018leopoldt}.

Let $\rho: \op{Gal}(\bar{\Q}/\Q) \to \GL_n(\mathbb{C})$ be a non-trivial irreducible Artin representation (which is not necessarily odd).
So it is a continuous, irreducible representation with finite non-trivial image.
Let $E$ be the finite extension of $\Q$ fixed by $\ker(\rho)$. Let $G = \op{Gal}(E/\Q)$ and $h_E$ be the class number of $E$.
Let $S$ be the set of primes of $\Q$ consisting of primes which are ramified in $E$ and $\infty$.

As $\op{Im}(\rho)$ is finite, there exists a number field $F$ such that, under a suitable basis, $\op{Im}(\rho) \subset \op{GL}_n(F)$.
Let $\cO_F$ be the ring of integers of $F$ and $\mathfrak{p}$ be a prime of $\cO_F$.
Let $F_{\mathfrak{p}}$ be the completion of $F$ at $\mathfrak{p}$, $\cO$ be its ring of integers and $\F$ be its residue field.
Let $p$ be the rational prime lying below $\mathfrak{p}$ and let $S_{\mathfrak{p}} := S \cup \{p\}$.
Then, under a suitable basis, the representation $\rho_{\mathfrak{p}} : G_{\Q,S_{\mathfrak{p}}} \to \op{GL}_n(F_{\mathfrak{p}})$, obtained by composing $\rho : G_{\Q,S_{\mathfrak{p}}} \to \op{GL}_n(F)$ with the map $F \to F_{\mathfrak{p}}$ induced by completion, takes values in $\op{GL}_n(\cO)$.

Recall that we have fixed $p$ to be an odd prime. Moreover, assume that $p \nmid |G|.h_E$, $p$ is unramified in $E$ and $p \geq \text{max} \{\ell^{|G|} \mid \ell \in S\}$.
Then, the residual representation $\bar{\rho}_{\mathfrak{p}} : G_{\Q,S_{\mathfrak{p}}} \to \GL_n(\F)$ obtained by reducing $\rho_{\mathfrak{p}} :G_{\Q,S_{\mathfrak{p}}} \to \op{GL}_n(\cO)$ modulo the maximal ideal of $\cO$ is absolutely irreducible (see \cite[Section 6.1]{khare2018leopoldt} for more details).
Thus it is natural to ask whether $H^2(\Q_{S_{\mathfrak{p}}}/\Q,\bar\rho_{\mathfrak{p}})=0$ for all but finitely many such primes $\mathfrak{p}$.
This question is studied in \cite{khare2018leopoldt}.

To be precise, in the setup described above, they prove:

\begin{theorem}[B\"{o}ckle-Guiraud-Kalyanswamy-Khare]
 Let $\rho : \op{Gal}(\bar{\Q}/\Q) \to \op{GL}_n(\mathbb{C})$ be a non-trivial irreducible Artin representation and let $F$ be a number field over which $\rho$ is defined as above. Let $\cO_F$ be the ring of integers of $F$.
Let $S$ be the set of primes of $\Q$ consisting of primes at which $\rho$ is ramified and $\infty$.
If $H^0(\op{Gal}(\mathbb{C}/\mathbb{R}),\rho)=0$, then $H^2(\Q_{S_{\mathfrak{p}}}/\Q,\bar\rho_{\mathfrak{p}}) = 0$ for all but finitely many primes $\mathfrak{p}$ of $\cO_F$.
\end{theorem}

\begin{proof}
This is a part of Proposition $6.6$ of \cite{khare2018leopoldt}.
\end{proof}

Combining this theorem with Theorem~\ref{theorem 3.6}, we get:
\begin{theorem}\label{artinthm}
Let $\rho : \op{Gal}(\bar{\Q}/\Q) \to \op{GL}_n(\mathbb{C})$ be a non-trivial irreducible Artin representation and let $F$ be a number field over which $\rho$ is defined as above. Let $\cO_F$ be the ring of integers of $F$.
Suppose $H^0(\op{Gal}(\mathbb{C}/\mathbb{R}),\rho)=0$.
Then, for all but finitely many primes $\mathfrak{p}$ of $\cO_F$, the following assertions hold:
\begin{enumerate}[(a)]
    \item the weak Leopoldt conjecture is true for $\rho_{\mathfrak{p}}$.
    \item The fine Selmer group $\Rfine{\mathbf{A}(\rho_{\mathfrak{p}}^*)}{\Q_{\op{cyc}}}$ associated to $\rho_{\mathfrak{p}}^*$ is cotorsion over $\Lambda$ with $\mufn{\rho_{\mathfrak{p}}^*}=0$.
\end{enumerate}
\end{theorem}

\begin{remark}
Note that, in \cite[Proposition 6.6]{khare2018leopoldt}, it is also proved that if $$0 < \dim(H^0(\op{Gal}(\mathbb{C}/\mathbb{R}),\rho)) < \dim(\rho)$$ and Heuristic $6.5$ of \cite{khare2018leopoldt} holds, then $H^2(\Q_{S_{\mathfrak{p}}}/\Q,\bar\rho_{\mathfrak{p}}) = 0$ for all but finitely many primes $\mathfrak{p}$ of $\cO_F$.
Combining this result with Theorem~\ref{theorem 3.6}, we conclude that $\mufn{\rho_{\mathfrak{p}}^*}=0$ for all but finitely many primes $\mathfrak{p}$ of $\cO_F$ in these cases as well.
\end{remark}

Note that the same question for Artin representations of arbitrary number fields is analyzed in Section $6.2$ of \cite{khare2018leopoldt}.

\subsection{Neat Representations}
\par Let $K$ be a number field, $\F$ be a finite field of characteristic $p$ and $S$ be a finite set of primes of $K$ containing all primes of $K$ dividing $p$ and all archimedean primes.
Let $\bar\rho : G_{K,S} \to \op{GL}_n(\F)$ be a continuous, absolutely irreducible representation.
Let $\cO$ be the ring of integers of a finite extension of $\Q_p$ with residue field $\F$.
Fix a uniformizer $\varpi$ of $\cO$.

Suppose $\bar\rho$ is unobstructed and let $d=\dim_{\F}(H^1(K_S/K,\op{Ad}\bar\rho))$.
Then, from \cite[Proposition 2]{mazur1989deforming}, it follows that the residual representation $\bar{\rho}$ lifts to the universal representation unramified away from $S$
\[\rho^{\op{univ}}: \op{G}_{K,S}\rightarrow \op{GL}_n\left(\cO\llbracket X_1, \cdots, X_d\rrbracket\right),\]which we view as a family of Galois representations. Indeed, given any $d$-tuple $a=(a_1, \cdots,a_d)\in \cO^d$, the $\cO$-valued homomorphism
$\varphi_a:\cO\llbracket X_1, \cdots, X_d\rrbracket\rightarrow \cO$ sending $X_i$ to $\varpi a_i$ gives rise to a Galois representation $\rho_a$ given by the composite
\begin{equation}\label{eqtn-a}
    \rho_a:\op{G}_{K,S}\xrightarrow{\rho^{\op{univ}}} \op{GL}_n\left(\cO\llbracket X_1, \cdots, X_d\rrbracket\right)\xrightarrow{\varphi_a^*} \op{GL}_n(\cO),\end{equation} where the second homomorphism is induced by $\varphi_a$. 
Since $\bar{\rho}$ is unobstructed, it follows from Theorem \ref{unobstructed implies mu is zero} that the weak Leopoldt conjecture holds for $\op{Ad} \rho_a$, and moreover, $\mufn{\op{Ad} \rho_a(1)}=0$. Thus, we get a family of Galois representations parametrized by $\cO^d$ for which the $\mu$-invariant of the fine Selmer group of the adjoint representation vanishes. 

\par Recall that if $K=\Q$ and $n=2$, then we get that $d=3$ if $\bar\rho$ is odd and $d=1$ if $\bar\rho$ is even (see Proposition~\ref{power series unobstructed}).

\par In \cite{mazur1989deforming}, Mazur introduced the notion of \emph{neatness} and constructed odd Galois representations of  $\op{Gal}(\bar{\Q}/\Q)$ that are unobstructed and unramified away from a single prime $p$. B\"{o}ckle \cite{bockle1999even} gave a similar construction for \emph{even}, unobstructed Galois representations of $\op{Gal}(\bar{\Q}/\Q)$. We relate the notion of neatness to vanishing of the $\mu$-invariant of fine Selmer groups. 

\par Fix an absolutely irreducible residual representation $\bar{\rho}:\op{G}_{K,S}\rightarrow \op{GL}_n(\F)$ as above, let $L/K$ be the splitting field cut out by $\bar{\rho}$. Denote by $S'$ the set of primes of $L$ that lie above the non-archimedean primes of $S$ and set $G$ to be the Galois group $G=\op{Gal}(L/K)$.
For a prime $w \in S'$, denote by $L_w$ the completion of $L$ at $w$.
Denote the subgroups of $L^\times$ and $L_w^\times$ consisting of of $p$-th roots of unity by $\mu_p(L)$ and $\mu_p(L_w)$, respectively.
Denote the ring of integers of $L$ and $L_w$ by $\cO_L$ and $\cO_{L_w}$, respectively.

\par An $\F_p[G]$-module $M$ is said to be \emph{coprime} to $\op{Ad}\bar{\rho}$ if the tensor product $M\otimes_{\F_p} (\op{Ad}\bar{\rho})^{\vee}$ does not contain the identity representation. 
\begin{definition}\label{mdef}
We define three Galois modules $M_1$, $M_2$ and $M_3$ associated to $L$ as follows:
\begin{enumerate}
    \item $M_1:=\op{coker}\left\{\mu_p(L)\rightarrow \bigoplus_{w\in S'} \mu_p(L_w)\right\}$;
    \item $M_2:=\op{ker}\left\{\cO_L^*/(\cO_L^*)^p\rightarrow \bigoplus_{w\in S'}\cO_{L_w}^*/(\cO_{L_w}^*)^p\right\}$;
    \item $M_3:=\op{Cl}(L)\otimes \F_p$, where $\op{Cl}(L)$ is the class group of $L$.
\end{enumerate}
\end{definition}
\begin{definition}[Mazur]
With notation as above, $\bar{\rho}$ is said to be \emph{neat} if $p$ does not divide the cardinality of $\op{Im}(\bar\rho)$ and $M_1$, $M_2$ and $M_3$ are coprime to $\op{Ad} \bar{\rho}$.
\end{definition}
See \cite[Section 1.12]{mazur1989deforming} for more details.
\begin{theorem}[Mazur]\label{neat implies unobstructed} Let $\F$ be a finite field of characteristic $p$, $K$ be a number field and $S$ be a set of primes of $K$ containing all primes above $p$ and all archimedean primes. Let $\bar{\rho}:\op{G}_{K,S}\rightarrow \op{GL}_n(\F)$ be a \emph{neat} residual representation. Then $\bar{\rho}$ is \emph{unobstructed}.  
\end{theorem}
\begin{proof}
The reader is referred to the proof of \cite[Proposition 7]{mazur1989deforming}.
\end{proof}

\begin{corollary}\label{corollary neat}
Let $\bar{\rho}$ be a neat Galois representation as in Theorem \ref{neat implies unobstructed}, let $d=\dim_{\F}(H^1(K_S/K,\op{Ad}\bar\rho))$ and let $\rho_a:\op{G}_{K, S}\rightarrow \op{GL}_n(\cO)$ be the $\cO$-deformation of $\bar\rho$ associated to $a=(a_1, \cdots, a_d)\in \cO^d$ as above (see \eqref{eqtn-a} for the definition of $\rho_a$). Then, the following assertions hold:
\begin{enumerate}
    \item the weak Leopoldt conjecture is true for $\op{Ad}\rho_a$,
    \item the fine Selmer group $\Rfine{\mathbf{A}(\op{Ad}\rho_a(1))}{K_{\op{cyc}}}$ associated to $\op{Ad}\rho_a(1)$ is cotorsion over $\Lambda$ with $\mufn{\op{Ad}\rho_a(1)}=0$. If $n=2$, then the same assertion holds for $\op{Sym}^2(\rho_a) \otimes (\det(\rho_a))^{-1}\chi_p$.
\end{enumerate}
\end{corollary}
\begin{proof}
The result is a direct consequence of Theorem \ref{neat implies unobstructed} and Theorem \ref{unobstructed implies mu is zero}.
\end{proof}

In section 1.13 of \emph{loc.cit.}, Mazur constructs examples of $2$-dimensional neat, odd Galois representations associated to certain $S_3$-extensions of $L/\Q$ ( which means $d=3$). The examples give a family of odd representations $\bar{\rho}$ for which the above Corollary applies. More specifically, let $p$ be a prime number which can be represented as $27+4a^3$, where $a$ is an integer. Let $L$ be the splitting field of $f(x):=x^3+ax+1$, the Galois group $\op{Gal}(L/\Q)$ is isomorphic to $S_3$. The discriminant of $f(x)$ is $-p$ and therefore, $p$ is the only prime that is ramified in $L$. There is a natural inclusion of $S_3$ into $\op{GL}_2(\F_p)$, via which we obtain a Galois representation $\bar{\rho}:\op{G}_{\Q,\{p,\infty\}}\rightarrow \op{GL}_2(\F_p)$. We refer to \emph{loc. cit.} for further details. Such Galois representations are referred to as \emph{special $S_3$-representations} since they are unobstructed, i.e., $H^2(\op{G}_{\Q,\{p,\infty\}}, \op{Ad}\bar{\rho})=0$. 
Since $\bar{\rho}$ is an $S_3$-representation, it is easy to see that $\bar{\rho}$ is a direct summand of $\op{Ad}\bar{\rho}$. Hence, it follows that $H^2(\op{G}_{\Q,\{p,\infty\}},{\bar{\rho}})=0$. 
Note that $\bar\rho$ is self-dual i.e. $\bar\rho^{\vee} = \bar\rho$.
Therefore, it follows from Theorem \ref{theorem 3.6} that for any representation $\rho:\op{G}_{\Q, \{p,\infty\}}\rightarrow \op{GL}_2(\Z_p)$ that lifts $\bar{\rho}$, the following assertions hold:
\begin{enumerate}
    \item the weak Leopoldt conjecture is true for $\rho$,
    \item the fine Selmer group $\Rfine{\mathbf{A}(\rho(1))}{\Q_{\op{cyc}}}$ associated to $\rho(1)$ is cotorsion over $\Lambda$ with $\mufn{\rho(1)}=0$.
\end{enumerate}

Similarly in \cite[Section 3.1]{bockle1999even}, B\"{o}ckle constructs examples of $2$-dimensional neat, even Galois representations associated to certain totally real $S_3$-extensions of $L/\Q$ (which means $d=1$).
The examples give a family of even representations $\bar{\rho}$ for which the above Corollary applies.
To be precise, let $p$ be a prime number which can be represented as $4a^6-27$, where $a \geq 2$ is an integer and $L$ be the splitting field of $f(x):=x^3-a^2x-1$. So the Galois group $\op{Gal}(L/\Q)$ is isomorphic to $S_3$. The discriminant of $f(x)$ is $p$ and therefore, $p$ is the only prime that is ramified in $L$. Using the arguments of the previous paragraph, we obtain a Galois representation $\bar{\rho}:\op{G}_{\Q,\{p,\infty\}}\rightarrow \op{GL}_2(\F_p)$. 
In \cite[Theorem 3.1]{bockle1999even}, B\"{o}ckle proves that $\bar\rho$ is neat and hence, unobstructed if the following hypotheses are satisfied:
\begin{enumerate}
    \item $p$ satisfies the Ankeney-Artin-Chowla conjecture,
    \item $\left(u\left(\dfrac{3}{2a^2}\right)-1\right)(2a^3+9)/6-\left(u\left(\dfrac{3}{3-2a^3}\right)-1\right) \not\equiv \dfrac{1}{243}(4a^6-27)(9/4+a^3) \pmod{(4a^6-27)^2}$, where for $x \in \Z_p$, $u(x):=x/x^p$.
\end{enumerate}
It is verified in \emph{loc.cit.} that these hypotheses are satisfied for all $108$ primes $p$ of the form $4a^6-27$ with $2 \leq a \leq 1000$ (see \cite[Section 3.1]{bockle1999even} for more details).
As noted above, $\bar{\rho}$ is a direct summand of $\op{Ad}\bar{\rho}$ in this case as well. Therefore, we conclude, using the same arguments as above, that for any representation $\rho:\op{G}_{\Q, \{p,\infty\}}\rightarrow \op{GL}_2(\Z_p)$ that lifts $\bar{\rho}$, the following assertions hold:
\begin{enumerate}
    \item the weak Leopoldt conjecture is true for $\rho$.
    \item The fine Selmer group $\Rfine{\mathbf{A}(\rho(1))}{\Q_{\op{cyc}}}$ associated to $\rho(1)$ is cotorsion over $\Lambda$ with $\mufn{\rho(1)}=0$.
\end{enumerate}

\section{Residually dihedral representations}\label{s 5}
\par In this section, we study residually dihedral Galois representations that arise from modular forms to illustrate Theorem \ref{th h2 vanishing}.

\par To be precise, let $p$ be an odd prime, $\F$ be a finite field of characteristic $p$, $S$ be a finite set of primes of $\Q$ containing $p$ and $\infty$ and $\bar\rho : G_{\Q,S} \to \op{GL}_2(\F)$ be a continuous representation.
We assume that $\bar\rho$ is dihedral, which is to say that there exists a quadratic extension $K$ of $\Q$ and a character $\psi:\op{G}_{K,S'}\rightarrow \F^\times$, where $S'$ is the set of primes of $K$ that lie above the primes in $S$, such that $\bar{\rho}:\op{G}_{\Q,S}\rightarrow \op{GL}_2(\F)$ is isomorphic to the induced representation $\op{Ind}_{\op{G}_{K,S'}}^{\op{G}_{\Q,S}} (\psi)$. 
For the ease of notation, we will denote this representation by $\op{Ind}_{\op{G}_{K}}^{\op{G}_{\Q}} (\psi)$.

\par Let $\cO'$ be the ring of integers of a finite extension of $\Q_p$ with residue field $\F$ and fix a uniformizer $\varpi$ of $\cO'$.
We say that a representation $\rho : G_{\Q,S} \to \op{GL}_2(\cO')$ is \emph{residually dihedral} if the residual representation $\bar\rho : G_{\Q,S} \to \op{GL}_2(\F)$ obtained by reducing $\rho$ modulo $\varpi$ is dihedral as described above.
Note that a residually dihedral representation $\rho$ is not necessarily dihedral.

\par Now suppose $\bar\rho : G_{\Q,S} \to \op{GL}_2(\F)$ is dihedral which means  $\bar\rho = \op{Ind}_{\op{G}_K}^{\op{G}_{\Q}} (\psi)$ for some character $\psi : G_{K,S'} \to \FF^{\times}$. Then the projective image of $\bar\rho$ (i.e. the image of $\bar\rho$ under the natural surjective map $\op{GL}_2(\F) \to \op{PGL}_2(\F)$ obtained by going modulo scalars) is either a dihedral group or a cyclic group.
Moreover, $\bar\rho$ is absolutely irreducible if and only if the projective image of $\bar\rho$ is a dihedral group.

\par In addition to being dihedral, now assume $\bar\rho$ is also odd and absolutely irreducible. Combining the work of Khare and Wintenberger on Serre's conjecture (\cite[Theorem 9.1]{KW}) and work of Kisin (\cite[Theorem 0.1]{Kisin}), we see that $\bar{\rho}$ lifts to a $p$-adic modular Galois representation $\rho_{f,\mathfrak{p}}:\op{G}_{\Q,S}\rightarrow \op{GL}_2(\cO)$ without increasing the set $S$ of ramified primes for the representation. There is however, always a natural choice of lift $\rho:\op{G}_{\Q,S}\rightarrow \op{GL}_2(W(\F))$ (where $W(\F)$ is the ring of Witt vectors of $\F$), letting $\tilde{\psi}$ denote the Teichm\"uller lift of $\psi$, we see that $\rho=\op{Ind}_{\op{G}_{K}}^{\op{G}_{\Q}}(\tilde{\psi})$ is a lift of $\bar{\rho}$.

\par In the setting of dihedral representations $\bar\rho$ as above, we will provide explicit criteria for the conditions in Theorem~\ref{th h2 vanishing} to be satisfied.
As a result, we obtain that if $\bar\rho$ satisfies these criteria, then for all characteristic $0$ lifts $\rho$ of $\bar\rho$, and in particular for the lifts arising from modular newforms, the weak Leopoldt conjecture holds for $\rho^*$ and $\mufn{\rho}=0$ (see Theorem~\ref{th 5.6}).
Note that lifts of such representations arising from CM modular forms have been studied in \cite{billereynuccio}.

\par

\begin{lemma}\label{dihedral lemma 1}
Let $\cO$ be the ring of integers of a finite extension of $\Q_p$ with residue field $\F$.
Let $\rho : G_{\Q,S} \to \op{GL}_2(\cO)$ be a Galois representation such that its residual representation $\bar{\rho}=\op{Ind}_{\op{G}_K}^{\op{G}_{\Q}} (\psi)$ is dihedral as above.
Let $L$ be the extension of $\Q$ cut out by $\bar\rho$ and let $G =\op{Gal}(L/\Q)$. Then $H^1(G, \mathbf{V}_{\bar{\rho}})=0$ (i.e. condition \eqref{c3 of th h2} in Theorem \ref{th h2 vanishing} is satisfied).
\end{lemma}
\begin{proof}
Observe that the order of $G$ is coprime to $p$. Therefore, it follows that $H^1(G, \mathbf{V}_{\bar{\rho}})=0$.
\end{proof}

Let $\bar{\rho}=\op{Ind}_{\op{G}_K}^{\op{G}_{\Q}} (\psi)$ and let $L$ be the field extension of $\Q$ which is fixed by the kernel of $\bar{\rho}$ as above. Note that $K\subset L$. 
Let $K_1$ and $L_1$ be the mod-$p$ Hilbert class fields of $K$ and $L$, respectively. In other words, $\op{Gal}(K_1/K)$ and $\op{Gal}(L_1/L)$ are identified with the mod-$p$ quotients of the class groups of $K$ and $L$, respectively.
Now $S$ is a finite set of primes of $\Q$ containing $p$, $\infty$ and the primes $\ell\neq p$ at which $\bar\rho$ is ramified. So $S$ contains the primes that ramify in $L$. 
Set $K_{1,S}$ (resp. $L_{1,S}$) to be the maximal subextension of $K_1$ (resp. $L_1$) in which the primes of $K$ (resp. $L$) above $S$ split completely. 
Thus, $\op{Gal}(K_{1,S}/K)$ (resp. $\op{Gal}(L_{1,S}/L)$) is identified with the maximal quotient of $\op{Cl}(K)/p\op{Cl}(K)$ (resp. $\op{Cl}(L)/p\op{Cl}(L)$) such that the primes of $K$ (resp. $L$) above $S$ split completely in the corresponding subextension. 
The fields defined fit into the following diagram:
\[
\begin{tikzpicture}[node distance = 1.5cm, auto]
      \node (K) {$K$} [above of=Q];
      \node (L) [above of=K, right of=K] {$L$};
      \node (KS) [above of=K] {$K_{1,S}$};
      \node (LS) [above of=L] {$L_{1,S}$};
      \node (K1) [above of=KS] {$K_{1}$};
      \node (L1) [above of=LS] {$L_{1}$};
      \draw[-] (K) to node {} (L);
      \draw[-] (K) to node {} (KS);
      \draw[-] (L) to node {} (LS);
      \draw[-] (KS) to node {} (LS);
      \draw[-] (KS) to node {} (K1);
      \draw[-] (LS) to node {} (L1);
      \draw[-] (K1) to node {} (L1);
      \end{tikzpicture}
\]
Note that all fields in the diagram are Galois over $\Q$ and that $G=\op{Gal}(L/\Q)$ acts naturally on $H_{L}:=\op{Gal}(L_1/L)$ and $H_{L,S} :=\op{Gal}(L_{1,S}/L)$.
Here, the action is described as follows. Given $v\in H_{L}$ and $\sigma\in G$, pick a lift $\tilde{\sigma}$ of $\sigma$ to $\op{Gal}(L_1/\Q)$. Since $H_{L}$ is a normal subgroup of $\op{Gal}(L_1/\Q)$ it follows that $\tilde{\sigma} v \tilde{\sigma}^{-1}$ belongs to $H_{L}$. Since $H_{L}$ is abelian, $\tilde{\sigma} v \tilde{\sigma}^{-1}$ is independent of the choice of lift $\tilde{\sigma}$ of $\sigma$. The action is defined by setting $\sigma\cdot v:=\tilde{\sigma} v \tilde{\sigma}^{-1}$ for any choice of lift $\tilde{\sigma}$ of $\sigma$. 
The module $H_{L,S}$ is a $G$-stable quotient of $H_{L}$, i.e., a quotient by a $G$-stable submodule.

\begin{lemma}\label{ lemma 5.2}
Suppose that $\bar{\rho}=\op{Ind}_{\op{G}_K}^{\op{G}_{\Q}} (\psi) : G_{\Q,S} \to \op{GL}_2(\F)$ is a dihedral representation and let $L$, $K_{1,S}$ and $L_{1,S}$ be the fields defined as above. 
Assume that $K_{1,S}\cdot L=L_{1,S}$. Let $G=\op{Gal}(L/\Q)$ and $M$ be an irreducible $\F_p[G]$-module such that $\op{dim} M>1$. Then, with respect to notation above, 
\[\op{Hom}_{G}(H_{L,S}, M)=0.\]
\end{lemma}
\begin{proof}
Since it is assumed that $K_{1,S}\cdot L=L_{1,S}$, it follows that the $G$-action on $H_{L,S}$ factors through an action of the quotient $\op{Gal}(K/\Q)$. 
Indeed, both $K_{1,S}$ and $L$ are abelian extensions of $K$ and hence, so is $L_{1,S}=L.K_{1,S}$.
Therefore, $\op{Gal}(L/K)$ acts trivially on $\op{Gal}(L_{1,S}/L)$.
As $\op{Gal}(K/\Q)$ is abelian, it follows that any irreducible representation of $\op{Gal}(K/\Q)$ is one-dimensional. The order of $G$ is coprime to $p$ and hence, it follows that $H_{L,S}$ decomposes into $1$-dimensional representations of $G$ that factor through the action of $\op{Gal}(K/\Q)$. Combining this with the assumption that $M$ is an irreducible $G$-module with $\op{dim} M>1$, we get that $\op{Hom}_{G}(H_{L,S}, M)=0$.
\end{proof}
\begin{lemma}\label{dihedral lemma 2}
Assume that $\bar\rho : G_{\Q,S} \to \op{GL}_2(\F)$ is dihedral i.e. $\bar{\rho}=\op{Ind}_{\op{G}_K}^{\op{G}_{\Q}} (\psi)$, and that its projective image is a dihedral group.
Let $L$, $K_{1,S}$ and $L_{1,S}$ be the fields defined as above and let $G = \op{Gal}(L/\Q)$.
Moreover, assume that the condition $K_{1,S}\cdot L=L_{1,S}$ of Lemma \ref{ lemma 5.2} is satisfied. Then, \[\op{Hom}_{G}(H_{L,S}, V_{\bar{\rho}})=0,\]i.e., condition \eqref{c1 of th h2} of Theorem \ref{th h2 vanishing} is satisfied for $\bar{\rho}$. 
\end{lemma}
\begin{proof}
Since the projective image of $\bar{\rho}$ is dihedral, it follows that ${\bar{\rho}}$ is an irreducible representation of $G$ of dimension $2$. Hence, the lemma follows directly from Lemma \ref{ lemma 5.2}.
\end{proof}
We shall provide some examples in which the condition $K_{1,S}\cdot L=L_{1,S}$ is satisfied. First, we obtain sufficient conditions for condition \eqref{c2 of th h2} in Theorem \ref{th h2 vanishing} to hold. Let $\nu$ be a generator of $\op{Gal}(K/\Q)$ and $\psi^{\nu}$ be defined as follows:
\[\psi^{\nu}(x):=\psi(\tilde{\nu} x \tilde{\nu}^{-1}),\] where $\tilde{\nu}$ is a lift of $\nu$ in $G_{\Q,S}$. Note that in the above formula, $\psi(\tilde{\nu} x \tilde{\nu}^{-1})$ is independent of the choice of the lift $\tilde{\nu}$ of $\nu$.

\begin{lemma}
Assume that $\bar{\rho}=\op{Ind}_{\op{G}_K}^{\op{G}_{\Q}} (\psi)$.
Then the projective image of $\bar{\rho}$ is a dihedral group if and only if $\psi\neq \psi^{\nu}$.
Moreover, the image of $\bar\rho$ is itself a non-abelian dihedral group if and only if $\psi^{-1}= \psi^{\nu}$ and $\psi^{\nu} \neq \psi$.
\end{lemma}
\begin{proof}
Recall that the projective image of $\bar\rho  =\op{Ind}_{\op{G}_K}^{\op{G}_{\Q}} (\psi)$ is dihedral if and only if $\bar\rho$ is irreducible.

Let $V$ be the vector space underlying $\bar\rho$.
As $\bar{\rho}=\op{Ind}_{\op{G}_K}^{\op{G}_{\Q}} (\psi)$, we get, after choosing a suitable basis of $V$, that $\bar\rho(g) = \begin{pmatrix}\psi(g) & 0 \\ 0 & \psi^{\nu}(g)\end{pmatrix}$ for all $g \in G_K$ and $\bar\rho(\tilde{\nu}) =\begin{pmatrix} 0 & 1 \\ a & 0\end{pmatrix}$ for some $a \in \mathbb{F}.$
Now $G_{\Q} = G_K \cup \tilde{\nu}G_K$.
Therefore, $\bar\rho$ is reducible if and only if there exists a non-zero $v \in V$ which is a common eigenvector for $\bar\rho(\tilde{\nu})$ and $\bar\rho(g)$ for all $g \in G_K$.
If $\psi = \psi^{\nu}$, then such an eigenvector clearly exists and hence, $\bar\rho$ is reducible.

If $\psi \neq \psi^{\nu}$, then there exists a $g_0 \in G_K$ such that $\psi(g_0) \neq \psi^{\nu}(g_0)$.
So if a non-zero vector $v \in V$ is an eigenvector for $\bar\rho(g_0)$, then either $\bar\rho(g_0)v =\psi(g_0)v$ or $\bar\rho(g_0)v =\psi^{\nu}(g_0)v$.
Let $w = \bar\rho(\tilde{\nu})v$. Then $w \neq 0$.
If $\bar\rho(g_0)v =\psi(g_0)v$, then $\bar\rho(g_0)w = \psi^{\nu}(g_0)w$.
As $\psi(g_0) \neq \psi^{\nu}(g_0)$ and $w \neq 0$, it follows that $w$ is not a scalar multiple of $v$. This means that $v$ is not an eigenvector of $\bar\rho(\tilde{\nu})$.
If $\bar\rho(g_0)v =\psi^{\nu}(g_0)v$, then we similarly conclude that $v$ is not an eigenvector of $\bar\rho(\tilde{\nu})$.
Therefore, if $\psi \neq \psi^{\nu}$, then $\bar\rho$ is irreducible.
This finishes the proof of the first part of the lemma.

Note that the image of $\bar\rho$ is a non-abelian dihedral group if and only if $\bar\rho(G_K)$ is  a cyclic group of order $> 2$ and $\bar\rho(\tilde{\nu}g\tilde{\nu}^{-1}) = \bar\rho(g^{-1})$ for all $g \in G_K$.
Further, $\bar\rho(\tilde{\nu}g\tilde{\nu}^{-1}) = \bar\rho(g^{-1})$ for all $g \in G_K$ if and only if $\psi^{\nu}=\psi^{-1}$. Moreover, if this holds, then $\bar\rho(G_K)$ is cyclic. If $\psi^{\nu}=\psi^{-1}$, then $\bar\rho(G_K)$ is cyclic of order $> 2$ if and only if $\psi^{-1} \neq \psi$.
Therefore, we conclude that the image of $\bar\rho$ is a non-abelian dihedral group if and only if $\psi^{\nu} =\psi^{-1}$ and $\psi^{\nu} \neq \psi$.
\end{proof}

\begin{lemma}\label{dihedral lemma 3}
Assume that $\bar{\rho}=\op{Ind}_{\op{G}_K}^{\op{G}_{\Q}} (\psi)$. 
Let $\ell$ be a prime and $L$ be the extension of $\Q$ cut out by $\ker(\bar\rho)$. Suppose that the trivial representation is a subrepresentation of $\bar{\rho}_{|\ell}$. Let $w \mid \ell$ be a choice of prime of $K$ that lies above $\ell$. Then, at least one of the following conditions hold:
\begin{enumerate}
    \item\label{5.4 cond 1} $\ell$ splits in $K$ and either $\psi_{|w}=1$ or $\psi^{\nu}_{|w}=1$,
    \item\label{5.4 cond 2} $\ell$ is either inert or ramified in $K$, and the unique prime $w$ of $K$ lying above $\ell$ splits completely in $L$.
\end{enumerate}

\end{lemma}
\begin{proof}
Assume that $\ell$ splits in $K$. Then, we have that $\bar{\rho}_{|w}=\mtx{\psi_{|w}}{0}{0}{\psi^{\nu}_{|w}}$, and the result follows in this case. 

Now suppose $\ell$ is either inert or ramified in $K$. 
Then $w$ is the unique prime of $K$ lying above $\ell$ and we can choose $\tilde\nu \in G_{\ell}$.
As the trivial representation is a subrepresentation of $\bar\rho_{|\ell}$, it follows that either $\psi_{|w}=1$ or $\psi^{\nu}_{|w}=1$.

As $\tilde\nu \in G_{\ell}$ and $G_w$ is normal in $G_{\ell}$, it follows that $\tilde\nu G_w \tilde\nu^{-1}= G_w$.
Therefore, it follows that if $\psi_{|w}=1$, then $\psi^{\nu}_{|w}=1$ and vice versa. This finishes the proof of the Lemma.


\end{proof}
Combining Lemmas \ref{dihedral lemma 1}, \ref{dihedral lemma 2} and \ref{dihedral lemma 3}, we obtain the following result.

\begin{theorem}\label{th 5.5}
Let $\cO$ be the ring of integers in a finite extension of $\Q_p$ with residue field $\F$ and $S$ be a set of primes of $\Q$ containing $p$ and $\infty$.
Consider a Galois representation $\rho:\op{G}_{\Q,S}\rightarrow \op{GL}_2(\cO)$ and let $\bar{\rho}$ be its residual representation. Assume that the following assertions hold:
\begin{enumerate}
    \item\label{5.5 cond 1} For some quadratic field extension $K/\Q$ and a character $\psi:G_{K,S'}\rightarrow \F^\times$ (where $S'$ is the set of primes of $K$ lying above primes in $S$), the residual representation $\bar{\rho}$ is isomorphic to the induced representation $\op{Ind}_{\op{G}_K}^{\op{G}_{\Q}}(\psi)$. 
    In addition, we assume that the projective image of $\bar{\rho}$ is a non-abelian dihedral group.
    \item\label{5.5 cond 2} With respect to above notation, we assume that $K_{1,S}\cdot L=L_{1,S}$ (cf. Lemma \ref{dihedral lemma 2}). 
    \item\label{5.5 cond 3} At each prime $\ell\in S$, the local representation $\bar{\rho}_{|\ell}$ does not have the trivial representation as a subrepresentation.
\end{enumerate}
Then, the following assertions hold:
\begin{enumerate}[(a)]
    \item the weak Leopoldt conjecture is true for $\rho^*$,
    \item the fine Selmer group $\Rfine{\mathbf{A}(\rho)}{\Q_{\op{cyc}}}$ associated to $\rho$ is cotorsion over $\Lambda$ with $\mufn{\rho}=0$.
\end{enumerate}
\end{theorem}
\begin{proof}
Note that condition~\eqref{c2 of th h2} of Theorem \ref{th h2 vanishing} is satisfied by assumption \eqref{5.5 cond 3} above.
It follows from Lemmas \ref{dihedral lemma 1} and \ref{dihedral lemma 2} that the remaining conditions of Theorem \ref{th h2 vanishing} are satisfied, and thus the result follows. 
\end{proof}

We reiterate that Lemma~\ref{dihedral lemma 3} gives sufficient conditions for condition~\eqref{5.5 cond 3} of Theorem~\ref{th 5.5} to hold.

Note that the above theorem involves conditions on the residual representation $\bar{\rho}$ and the set of primes $S$. Via purely Galois theoretic computations, we construct a class of examples of residual representations $\bar{\rho}:\op{G}_{\Q,S}\rightarrow \op{GL}_2(\F)$ satisfying the above conditions.

\par Note that the condition $K_{1,S}\cdot L=L_{1,S}$ is satisfied if $K_1\cdot L=L_1$. Examples satisfying this condition are constructed in \cite[section 5.2]{deo2020dihedral} (see \cite[Table 1]{deo2020dihedral}).
We will adapt the approach of \cite[Section 5.2]{deo2020dihedral} to our context to obtain a class of examples satisfying the conditions of Theorem~\ref{th 5.5}.
We will now briefly outline our strategy to get examples.

\par Let $K$ be an imaginary quadratic extension of $\Q$ and $p$ be an odd prime. Assume that there is an odd prime $q\neq p$ such that $q$ divides the class number of $K$. 
Let $N$ be the maximal extension of $K$ contained in the Hilbert class field of $K$ with odd degree of extension $[N:K]$.
By class field theory, $N$ is a Galois extension of $K$ and $\op{Gal}(N/K)$ is identified with the odd part of $\op{Cl}(K)$ i.e. with $\op{Cl}(K)/\op{Cl}_2(K)$, where $\op{Cl}_2(K)$ is the $2$-Sylow subgroup of $\op{Cl}(K)$.
Note that $\op{Cl}_2(K)$ is stable under the action of $\op{Gal}(K/\Q)$ on $\op{Cl}(K)$.
Therefore, using class field theory, we conclude that $N$ is Galois over $\Q$.

\par Let $\nu$ be a lift of the generator of $\op{Gal}(K/\Q)$ in $G_0=\op{Gal}(N/\Q)$ with order $2$.
Since $|\op{Gal}(N/K)|$ is odd, such a lift does exist. 
Let $H_0=\op{Gal}(N/K)$.
\begin{lemma}\label{cftlemma}
Under the notation and set-up established above, $\nu g \nu^{-1} = g^{-1}$ for all $g \in H_0$.
\end{lemma}
\begin{proof}
As $H_0$ is normal in $G_0$, $\nu H_0 \nu^{-1} = H_0$.
Let $H_1 = \{h \in H_0 \mid \nu h \nu^{-1} = h\}$ and $H_2 = \{h \in H_0 \mid \nu h \nu^{-1} = h^{-1}\}$.
So $H_1$ and $H_2$ are subgroups of $H_0$ and moreover, they are normal in $G_0$. As $H_0$ has odd order, we have $H_1 \cap H_2 = \{1\}$.
We now claim that $H_0=H_1H_2$.

Let $g \in H_0$ and suppose $\nu g \nu^{-1}=g' \in H_0$.
Let $g_1=gg'$ and $g_2=g(g')^{-1}$. So $g^2=g_1g_2$.
As $\nu^2=1$, we get that $\nu g_1 \nu^{-1} = g_1$ and $\nu g_2 \nu^{-1} = g_2^{-1}$.
So $g_1 \in H_1$ and $g_2 \in H_2$.
Thus $g^2 \in H_1H_2$ for all $g \in H_0$.
As $H_0$ is abelian and has odd order, every element of $H_0$ is a square and hence, $H_0 = H_1H_2$.

Suppose $|H_1|=k >1$. Note that $G' := G_0/H_2$ is an abelian group of order $2|H_1|=2k$. 
So if $N' \subset N_0$ is the fixed field of $H_2$, then $N'$ is an abelian extension of $\Q$ with $\op{Gal}(N'/\Q)=G'$. 
Note that $N'$ is an unramified extension of $K$.
Therefore, a prime $\ell$ of $\Q$ is ramified in $N'$ if and only if it is ramified in $K$.
As $K$ is quadratic over $\Q$, the order of the inertia group at any rational prime $\ell$ in $\op{Gal}(N'/\Q)$ is at most $2$.

As $k$ is odd and $G'$ is abelian, $G'$ contains a unique subgroup $H'$ of order $2$.
Thus the subfield of $N'$ fixed by $H'$ is an unramified abelian extension of $\Q$ of degree $k >1$. This gives us a contradiction.
Hence, $|H_1|=1$ which means $H_0=H_2$ proving the lemma.
\end{proof}


\par Let $L'$ be a cyclic, unramified extension of $K$ with $[L':K]=n>1$, $n$ odd (i.e. $L' \subset N$) and $p \nmid n$.
Note that such an extension exists as we are assuming that the class number of $K$ is divisible by an odd prime $q \neq p$.
From Lemma~\ref{cftlemma}, we conclude that $L'$ is Galois over $\Q$, with $\op{Gal}(L'/\Q) = D_{n}$, the dihedral group with $2n$ elements.
As $n$ is odd, $\op{Gal}(L'/\Q)$ is non-abelian.

\begin{theorem}\label{th 5.6}

 Let $\F$ be a field of characteristic $p$. Let $K$ be an imaginary quadratic field with class number divisible by an odd prime different from $p$.
 Let $L'$ be a cyclic extension of $K$ contained in the Hilbert class field of $K$ with $[L':K]=n>1$, $n$ odd and $p \nmid n$.
 Let $\psi:\op{Gal}(L'/K)\rightarrow \F^\times $ and $\eta : G_{\Q} \to \FF^{\times}$ be non-trivial characters. 
 Set $\bar{\rho}$ to denote the induced representation $\op{Ind}_{\op{G}_K}^{\op{G}_{\Q}}(\psi\eta)$ of $G_{\Q}$.
 Let $L$ be the extension of $\Q$ cut out by $\bar\rho$.
 Let $S$ be a set of prime numbers containing $p$, $\infty$ and the primes which are ramified in $L$.
 Let $\rho:\op{G}_{\Q,S}\rightarrow \op{GL}_2(\cO)$ be a lift of $\bar{\rho}$.
 Assume that the following conditions are satisfied:
 
 \begin{enumerate}
     \item\label{5.6 cond 1} If $\ell \in S$, $\ell$ is split in $K$ and $w$ is a fixed place of $K$ lying above $\ell$, then $|\psi(G_w)| \neq |\eta(G_{\ell})|$,
     \item\label{5.6 cond 2} If $\ell$ is either inert or ramified in $K$ and $w$ is the unique place of $K$ lying above $\ell$, then $\eta_{|\ell}$ is non-trivial and is not the quadratic character of $G_{\ell}$ corresponding to $K_w$,
     \item\label{5.6 cond 5} $L_{1,S}=L \cdot K_{1,S}$.
 \end{enumerate}
 
Then, the following assertions hold:
\begin{enumerate}[(a)]
    \item the weak Leopoldt conjecture is true for $\rho^*$,
    \item the fine Selmer group $\Rfine{\mathbf{A}(\rho)}{\Q_{\op{cyc}}}$ associated to $\rho$ is cotorsion over $\Lambda$ with $\mufn{\rho}=0$.
\end{enumerate}

\end{theorem}


\begin{proof}
Note that, under the hypotheses of the theorem above, conditions~\eqref{5.5 cond 1} and \eqref{5.5 cond 2} of Theorem~\ref{th 5.5} hold. 
Therefore, by Theorem~\ref{th 5.5}, it suffices to check that \emph{no} condition of Lemma~\ref{dihedral lemma 3} is satisfied.
Let $\ell \in S$ and fix a place $w$ of $K$ lying above $\ell$.

Suppose $\ell$ is split in $K$. Now condition \eqref{5.6 cond 1} implies that $\psi\eta_{|w} \neq 1$.
 As $\op{Gal}(L'/\Q)$ is a non-abelian dihedral group, it follows that $\psi^{\nu} = \psi^{-1}$.
As $\eta$ is a character of $G_{\Q}$, we get that $\eta^{\nu} = \eta$.
Therefore, condition \eqref{5.6 cond 1} implies that $(\psi\eta)^{\nu}_{|w} = \psi^{-1}\eta_{|w} \neq 1 $.
Thus condition \eqref{5.4 cond 1} of Lemma~\ref{dihedral lemma 3} is not satisfied.

Now suppose $\ell$ is either inert or ramified in $K$. As $L'$ is an unramified extension of $K$ and $K_w$ is a quadratic extension of $\Q_{\ell}$, it follows that the image of $G_{\ell}$ in $\op{Gal}(L'/\Q)$ is abelian.
As $\op{Gal}(L'/\Q) = D_{n}$ with $n$ odd, it follows that the image of $G_{\ell}$ in $\op{Gal}(L'/\Q)$ is an abelian group of order $2$.
Hence, $w$ is totally split in $L'/K$.
Now as $\eta_{|\ell}$ is not the quadratic character corresponding to $K_w$ and $\eta_{|\ell} \neq 1$, it follows that $\eta_{|w} \neq 1$ and hence, $w$ is not completely split in $L/K$.
Therefore, condition \eqref{5.4 cond 2} of Lemma~\ref{dihedral lemma 3} is not satisfied. This finishes the proof of the theorem.
\end{proof}

{\bf Examples:} Let $p=3$ and $K= \Q(\sqrt{-D})$ with $$D \in \{239, 971, 1259, 2243, 2699, 2843\}.$$ We verify, using LMFDB \cite{cremona2021functions}, that the class number of $K$ is $15$ (see \href{https://www.lmfdb.org/NumberField/?hst=List&degree=2&class_number=15&ram_quantifier=exactly&search_type=List}{Imaginary quadratic number fields with class number 15}).
Let $L'$ be the unramified, abelian extension of $K$ of degree $5$.
Let $\eta$ be the mod $3$ cyclotomic character of $\bar\chi_3$ of $G_{\Q}$.
So $L = L'(\mu_3)$.
We verify, from LMFDB \cite{cremona2021functions}, that $L_1 = L \cdot K_1$ which means condition~\eqref{5.6 cond 5} of Theorem~\ref{th 5.6} is satisfied for all finite set $S$ of primes of $\Q$.
Let $$\mathcal{S}_1 = \{\ell \text{ prime } \mid \ell \equiv 1 \pmod{3} \text{ and } \ell \text{ is split in } K \text{ but not totally split in } L'\}$$ and
$$\mathcal{S}_2 = \{\ell \text{ prime } \mid \ell \equiv -1 \pmod{3} \text{ and } \ell \text{ is not inert in } K\}.$$ 
Let $\mathcal{S} = \mathcal{S}_1 \cup \mathcal{S}_2$.
Note that, by Chebotarev density theorem, $\mathcal{S}$ is an infinite set with Dirichlet density $9/20$.
Let $S'$ be a finite subset of $\mathcal{S}$ and let $S = S' \cup \{3,D,\infty\}$.
Then $S$ satisfies conditions \eqref{5.6 cond 1} and \eqref{5.6 cond 2} of Theorem~\ref{th 5.6}.
Let $\F$ be the finite field with $81$ elements and let $\psi:\op{Gal}(L/K)\rightarrow \F^\times $ be a non-trivial character. 
 Let $\bar{\rho} : G_{\Q,S} \to \op{GL}_2(\F)$ be the induced representation $\op{Ind}_{\op{G}_K}^{\op{G}_{\Q}}(\psi\bar\chi_3)$ of $G_{\Q,S}$.
 If $\rho : G_{\Q,S} \to \op{GL}_2(\cO)$ is a lift of $\bar\rho$, then by Theorem~\ref{th 5.6}, we conclude that
\begin{enumerate}[(a)]
    \item the weak Leopoldt conjecture is true for $\rho^*$,
    \item the fine Selmer group $\mathcal{R}_{3^{\infty}}(\mathbf{A}(\rho)/{\Q_{\op{cyc}}})$ associated to $\rho$ is cotorsion over $\Lambda$ with $\mufn{\rho}=0$.
\end{enumerate}
Moreover, observe that $\bar\rho$ is self-dual i.e. $\bar\rho^{\vee}=\bar\rho$. 
Hence, we conclude that the weak Leopoldt conjecture is true for $\rho(1)$.

\bibliographystyle{alpha}
\bibliography{references}
\end{document}